\documentclass[11pt]{article}
\usepackage{amsmath,amssymb,amsthm}
\usepackage[pdftex]{graphicx}
\usepackage{a4}
\usepackage{url}
\usepackage{color}
\usepackage[urlcolor=black]{hyperref}
\usepackage{algorithm,algpseudocode}
\usepackage{bm}
\usepackage{enumitem}
\usepackage{nicefrac}
\usepackage{todonotes}
\usepackage[T1]{fontenc}
\usepackage[nocompress]{cite}
\usepackage{csvsimple}
\usepackage{mathabx}
\usepackage{geometry}
\usepackage{pgfplots}
\pgfplotsset{compat=newest}       
\usepgfplotslibrary{groupplots}
\usepackage{pgfplotstable}
\newcounter{tikzsubfigcounter}[figure]
\renewcommand{\thetikzsubfigcounter}{\the\numexpr\value{figure}+1\relax\alph{tikzsubfigcounter}}

\newcounter{tikzsubfigcounterinvisible}[figure]
\renewcommand{\thetikzsubfigcounterinvisible}{\the\numexpr\value{figure}+1\relax\alph{tikzsubfigcounterinvisible}}

\newcommand{\settikzlabel}[1]{ %
	\refstepcounter{tikzsubfigcounterinvisible} \label{#1} 
}

\usetikzlibrary{arrows,shapes,positioning}
\usetikzlibrary{decorations.markings}
\tikzstyle arrowstyle=[scale=1]
\tikzstyle directed=[postaction={decorate,decoration={markings,
		mark=at position .65 with {\arrow[arrowstyle]{stealth}}}}]
\tikzstyle reverse directed=[postaction={decorate,decoration={markings,
		mark=at position .65 with {\arrowreversed[arrowstyle]{stealth};}}}]
\usepackage[titletoc,title]{appendix}

\input{macros.sty}

\begin{document}

\title{A posteriori error analysis and adaptive non-intrusive numerical schemes for systems of random conservation laws}
\author{Jan Giesselmann\thanks{Department of Mathematics, TU Darmstadt, Dolivostra\ss e 15,  64293 Darmstadt, Germany. 
} 
\and Fabian Meyer\thanks{Institute of Applied Analysis and Numerical Simulation, University of Stuttgart, 
Pfaffenwaldring 57, 70569 Stuttgart, Germany. 
\newline F.M., C.R. thank the Baden-W{\"u}rttemberg Stiftung for support via the project 'SEAL'. J.G. thanks the German Research Foundation (DFG) for support of the project via DFG grant GI1131/1-1. $^*$fabian.meyer@mathematik.uni-stuttgart.de}~$^{,*}$
\and Christian Rohde\samethanks }
\date{\today}
\providecommand{\keywords}[1]{{\textit{Key words:}} #1\\ \\}
\providecommand{\class}[1]{{\textit{AMS subject classifications:}} #1}
\maketitle

\begin{abstract}  \noindent
In this article we consider  one-dimensional random systems of
hyperbolic conservation laws. We first establish existence and uniqueness of random 
entropy admissible solutions for initial value problems of conservation laws which
involve random initial data and  random flux functions.
Based on these results we present an a posteriori error analysis for a numerical approximation 
of the random entropy admissible solution. 
For the stochastic discretization, we consider a non-intrusive approach, the 
Stochastic Collocation method. The spatio-temporal discretization 
relies on the Runge--Kutta Discontinuous Galerkin method.
We derive the a posteriori estimator
using continuous reconstructions of the discrete solution. Combined with the relative entropy stability framework
this yields computable error bounds for the entire space-stochastic discretization error.
The estimator admits a splitting into a stochastic and a deterministic (space-time) part, allowing
for a novel residual-based space-stochastic adaptive mesh refinement algorithm.
We conclude with various numerical examples investigating the scaling properties of
the residuals and
illustrating the efficiency of the proposed adaptive algorithm.

\end{abstract}
\class{Primary 35L65, 35R60; Secondary 65M15, 65M60, 65M70}

\section{Introduction}
Quantifying the influence of random model parameters as well as uncertain initial or boundary conditions
has become an important task in computational science and engineering. Uncertainty Quantification (UQ) addresses 
this issue and provides a variety of mathematical methods to examine the influence of uncertain input parameters on 
numerical solutions and derived quantities of interest.

In this article we study (spatially) one-dimensional systems of random hyperbolic conservation laws, where the uncertainty
stems from random initial data or from  random flux functions. 
Based on Kru\v{z}kov's work \cite{kruvzkov1970}, a firm theory for random scalar conservation laws 
in several space dimensions has been developed \cite{TokarevaMishra2016, RisebroSchwab2016, MishraSchwab2012}.
Compared to scalar equations, little is known about existence and uniqueness 
of entropy solutions for systems of hyperbolic conservation laws, 
especially in multiple space dimensions. For one-dimensional systems with initial data with small total variation,
Glimm provided a proof for the global existence of entropy admissible solutions  \cite{Glimm1975}.
Later,  Bressan and coauthors \cite{BressanLeFloch1997,BressanLewicka2000,Bressan2001} proved that the entropy admissible solutions constructed by the Glimm scheme (or equivalently by wave-front tracking) are unique in the 
sense that they are the only entropy admissible solutions 
satisfying additional stability properties such as
certain bounds on the growth of their total variation. 
Based on these deterministic results, we prove the existence and uniqueness of so-called random entropy admissible solutions for one-dimensional systems of
random hyperbolic conservation laws (\thmref{thm:randomEntropySol}).

Having established existence and uniqueness we
approximate the random entropy admissible solution numerically.
We discretize the random space by the Stochastic Collocation (SC) method \cite{abgrallmishra2017, HesthavenXiu2005,BabuskaNobileTempone2010}. 
The method is non-intrusive, which means that the underlying deterministic solver does not need 
to be modified. Moreover, it is easily parallelizable and it avoids the problem 
of losing hyperbolicity for nonlinear hyperbolic systems,
 a major drawback of many intrusive methods, most notably the Stochastic Galerkin method
\cite{PoetteDespresLucor2009}. 
As specific deterministic solver we use the Runge--Kutta Discontinuous Galerkin method \cite{CockburnShu2001}.

For nonlinear random hyperbolic conservation laws, discontinuities, both in physical
and stochastic space, may appear in finite time.
It is therefore favorable to locally increase the spatial and stochastic mesh resolution around 
the discontinuities in physical and stochastic space.
Adaptive algorithms for the Stochastic Collocation method and for the
Simplex Stochastic Collocation method
have been considered 
in \cite{GunzburgerWebsterZhang2014,WittevennIaccarino2013}. 
The refinement criteria for these methods
are based on heuristic considerations and are not immediately linked to the true numerical error.
Moreover, they do not consider  refinement in physical space.
An approach which combines both, physical and stochastic refinement, has 
been introduced in \cite{BryantPrudhomme2015}, where 
the authors consider random boundary value problems
for second order partial differential equations and use adjoint methods to derive separable error bounds
for linear quantities of interest. They then use the corresponding residuals for local mesh refinement. 
For random elliptic problems the analysis of adaptive mesh refinements based on 
a posteriori error estimates 
is more advanced
than for random hyperbolic conservation laws, cf. \cite{GuignardNobile2018,SchiecheLang2014} 
and references therein.

In this work, we present as the first new contribution an a posteriori error analysis 
which is based on the following approach \cite{Makridakis2006}:
We view the numerical solution (or, more precisely, a reconstruction thereof) as the exact solution of a perturbed version of the original problem.~The perturbation is given by a computable residual which acts as a source term.~By using the appropriate stability theory for the problem at hand we can bound 
the difference between the exact and the numerical solution
in terms of the residual. The suitable stability theory for systems of hyperbolic conservation laws is the relative entropy stability framework of Dafermos and DiPerna, see \cite[Theorem 5.2.1 ]{Dafermos2016}.

The specific reconstructions, which we use, are based on reconstructions for deterministic problems suggested in \cite{GiesselmannDedner16}, 
see also \cite{GiesselmannMeyerRohde18} for a modified reconstruction 
in terms of Stochastic Galerkin schemes.
Using these reconstructions, we obtain a residual admitting a decomposition into a spatial and a stochastic part,
which enables us to control the errors arising from spatial and stochastic discretization.
Based on the a posteriori error estimate we exploit the residuals' structure and propose
as the second novel contribution of this paper a residual-based space-stochastic adaptive numerical scheme.
While for smooth solutions the estimator provides a reliable a posteriori error control, for discontinuous solutions it blows up under mesh refinement.  
However, the residuals precisely capture the positions of rarefaction waves, contact discontinuities and shocks. Therefore,
our residual-based space-stochastic adaptive numerical scheme
leads to a significant error reduction
compared to uniform mesh refinement.
Due to its non-intrusive structure our proposed method admits a straightforward parallelization, 
the residuals are 
on-the-fly computable and the resulting adaptive schemes efficiently decrease the numerical error
compared to uniform mesh refinement.

This article is structured as follows: In \secref{sec:prelim} we describe our equation of interest. 
In \secref{sec:rcls} we first review the deterministic well-posedness theory for one-dimensional systems of hyperbolic conservation laws.
We then introduce the notion of random entropy admissible solutions and establish existence and uniqueness under suitable
assumptions on the random initial data and random flux function. \secref{sec:discretization} 
describes the
stochastic discretization and we show how to obtain the reconstruction from our numerical solution.
In \secref{sec:aposteriori} we establish the a posteriori estimate and derive the splitting of the error estimator
into a spatio-temporal and a stochastic part. Furthermore, we describe our space-stochastic adaptive numerical schemes.
Finally, in \secref{sec:numerics}, we provide various numerical examples for the Euler equations, where on the one hand we examine the scaling behavior 
of the corresponding residuals and on the other hand we compare the error reduction of our adaptive numerical algorithms 
to that of uniform mesh refinements and show that our adaptive schemes are indeed more efficient.

\section{Preliminaries and Notation} \label{sec:prelim}
\subsection{A Primer on Probability Theory}
Let $(\Omega, \F,\P)$ be a probability space, where $\Omega$ is the set of all elementary events $\omega \in \Omega$, $\F$ is a $\sigma$-algebra on $\Omega$ and $\P$ is a probability measure. 
In the following we  parametrize the uncertainty with a random vector
 $\xi : \Omega\to  \Xi \subset \R^N$ with independent, absolutely continuous components, i.e. $\xi(\omega) = \big(\xi_1(\omega),\ldots, \xi_N(\omega)\big) : \Omega\to \Xi \subset \R^N $.
This means that for every random variable $\xi_i$ there exists a density function 
$\densityfunctionOneD{i} :\R \to [0,\infty)$, such that $\int \limits_{\R} \densityfunctionOneD{i}(y)~ \mathrm{d}y=1$ and  $\P[ \xi_i \in A]=\int \limits_A \densityfunctionOneD{i}(y)~ \mathrm{d}y$, 
for any $A \in \mathcal{B}(\R)$, for all $i=1,\ldots, N$. Here $\mathcal{B}(\R)$ is the
Borel $\sigma$-algebra on $\R$.
Moreover, the joint density function $w$ of the random vector
 $\xi = (\xi_1,\ldots,\xi_N)$ can be written as
$$ \densityfunction(y) = \prod \limits_{i=1}^N \densityfunctionOneD{i}(y_i)\quad \forall ~ y= (y_1,\ldots, y_N)^\top \in \Xi. $$
The random vector induces a probability measure $\tilde{\P}(B):= \P( \xi^{-1}(B))$ for all $B \in \mathcal{B}(\Xi)$ on the measurable space $(\Xi, \mathcal{B}(\Xi))$.
This measure is called the law of $\xi$ and in the following we work on the 
image probability space $(\Xi, \mathcal{B}(\Xi), \tilde{\P})$.

For a Banach space $E$ and its Borel $\sigma$-algebra $\mathcal{B}(E)$, we consider the weighted $L_{\densityfunction}^p$-spaces equipped with the norms
\begin{align*}
\|f\|_{L_w^p(\Xi;E)} :=
\begin{cases}
\Big(\int \limits_\Xi\|f(y)\|_E^p ~\densityfunction(y)\mathrm{d}y\Big)^{1/p} = \E\Big(\|f\|_E^p\Big)^{1/p},\quad& 1\leq p<\infty \\
\operatorname{ess~sup}\limits_{y \in \Xi} \|f(y)\|_E,  &p= \infty .
\end{cases} 
\end{align*}

\subsection{Statement of the Problem}
We start with the following one-dimensional hyperbolic system of $m \in \N$ (deterministic) conservation laws.
\begin{equation} \label{eq:IVP}   \tag{IVP}
\begin{aligned} 
\begin{cases}
\partial_t \stochSol(t,x) + \partial_x  \Flux(\stochSol(t,x))= 0 , &(t,x)  \in   (0,T) \times \R,
\\ \stochSol(0,x)= \stochSol^0(x),  & x \in  \R .\notag
\end{cases}
\end{aligned} 
\end{equation} 
Here, $\stochSol(t,x) \in \U \subset \R^m$ is the vector of conserved quantities, 
$\Flux \in C^2(\U;\R^m)$ is the flux function, and $\mathcal{U}\subset \R^m$ is the state space, which is assumed to be an open set and  $T \in (0,\infty)$. We make the following assumptions on
the initial condition and flux function. 
\begin{enumerate}[label=(D\arabic*),ref=D\arabic*]
  \item The initial condition satisfies $\stochSol^0\in {\Leb^1(\R;\U)}.$ 
        \label{ass:initialdata}
  \item  The Jacobian $\D\Flux$ has $m$ distinct real eigenvalues, with each characteristic field being
   either genuinely nonlinear or linearly degenerate \cite[Def. 3.1.1, Def. 7.5.1]{Dafermos2016}.
    \label{ass:flux}
\end{enumerate}

For our probabilistic equation of interest we admit in \eqref{eq:IVP} random variations 
in the flux and initial datum, 
leading to the following one-dimensional system of $m \in \N$ random conservation laws.
\begin{equation} \label{eq:RIVP}   \tag{RIVP}
\begin{cases}
\partial_t \stochSol(t,x,y) + \partial_x  \Flux(\stochSol(t,x,y),y)= 0, &(t,x,y)  \in   (0,T) \times \R \times \Xi,
\\ \stochSol(0,x,y)= \stochSol^0(x,y),  &(x,y) \in  \R \times  \Xi.\notag
\end{cases}
\end{equation} 
Here, $\stochSol(t,x,y) \in \U \subset \R^m$ is the random vector of conserved quantities and
$\Flux(\cdot,y)\in C^2(\U;\R^m)$, $\Pas$
is the random flux function. 

For the sake of simplicity we keep the same notation for the
solution and for the flux as in \eqref{eq:IVP} and make the following assumptions on the random initial condition and on the random flux function. 
Note that these assumption are the probabilistic versions of assumptions \eqref{ass:initialdata} and
\eqref{ass:flux}.
\begin{enumerate}[label=(R\arabic*),ref=R\arabic*]
  \item The uncertain initial condition satisfies $\stochSol^0\in \leb{1}{\Xi}{\Leb^1(\R;\U)}.$ 
        \label{ass:randominitialdata}
  \item For almost every realization $y\in \Xi$, the Jacobian $\D \Flux(\cdot,y)$ has $m$ distinct real
   eigenvalues, 
   and each characteristic field is 
    either linearly degenerate 
  or genuinely nonlinear. 
  Moreover, we assume that $\Flux \in \leb{2}{\Xi}{C^2(\U;\R^m)}$.  \label{ass:randomflux}
\end{enumerate}

\section{Well-Posedness: Deterministic vs. Random Hyperbolic Conservation Laws} \label{sec:rcls}
In this section we first review some classical results for deterministic one-dimensional hyperbolic conservation laws.
We then introduce the notion of a random entropy solution for \eqref{eq:RIVP} and establish its existence and uniqueness based on the results for the deterministic hyperbolic conservation law
\eqref{eq:IVP}. 
\subsection{Deterministic Hyperbolic Conservation Laws} \label{subsec:detCL}
Let us consider the deterministic initial value problem \eqref{eq:IVP}.
We say that a strictly convex function $\eta \in C^2(\U;\R)$ and a function $q \in C^2(\U;\R)$
form an entropy/entropy-flux pair, if they satisfy $\D \eta \D \Flux =\D q$.
We assume that the system \eqref{eq:IVP} is endowed with at least one entropy/entropy-flux pair. 
We then define entropy admissible solutions in the following way.

\begin{definition}[Entropy admissible solution]
A function $\stochSol \in L^1((0,T)\times \R;\U)$ is called an entropy admissible solution of \eqref{eq:IVP}, if it satisfies the following conditions:
\begin{enumerate}
\item It is a weak solution, i.e.
\begin{align} \label{eq:entropySol}
\int \limits_0^T \int \limits_{\R} \Big( \stochSol(t,x) \cdot \partial_t \phi(t,x) & + \Flux(\stochSol(t,x)) \cdot  \partial_x\phi(t,x)\Big)~\mathrm{d}x \mathrm{d}t  \notag
\\ &=  - \int \limits_{\R} \stochSol^0(x) \cdot \phi(0,x)~ \mathrm{d}x,
\end{align}
for all $\phi\in C_c^\infty([0,T)\times \R;\R^m)$.
\item It satisfies the weak entropy inequality:
\begin{align} \label{ineq:entropySol}
\int \limits_0^T \int \limits_{\R} \Big( \eta(\stochSol(t,x)) \partial_t \Phi(t,x)  & +q(\stochSol(t,x))  \partial_x\Phi(t,x)\Big)~\mathrm{d}x \mathrm{d}t   \notag
\\ & \geq- \int \limits_{\R} \eta(\stochSol^0(x)) \Phi(0,x)~ \mathrm{d}x, 
\end{align}
for all $\Phi\in C_c^\infty([0,T)\times \R;\R_+)$.
\end{enumerate}
\end{definition}

\begin{theorem}[\hspace{-1pt}\cite{BressanLeFloch1997}, Theorem 2] \label{thm:entropySol} 
Provided \eqref{ass:flux} holds, there exists  a non-empty closed domain $\cD \subset \lebnonweight{1}{\R}{\U}$ of integrable functions with small total
variation and a semi-group $\cS(t): [0,\infty) \times \cD \to \cD$, called Standard Riemann Semigroup (SRS), that is unique (up to its domain) and
which has in particular the following properties:
\begin{enumerate}[label=(\roman*)]
\item There exists a constant $L>0$, such that 
\begin{align*}
\|\cS(s)\overline{u}- &\cS(t)\overline{v}\|_{\lebnonweight{1}{\R}{\R^m}}
  \leq L \Big(|s-t|+ \|\overline{u}- \overline{v}\|_{\lebnonweight{1}{\R}{\R^m}}\Big),
\end{align*}
for all $\overline{u},\overline{v} \in \cD$ and for all $s,t\geq 0$. 
\item For $\overline{u}\in \cD$ 
the function $u(t,x):= (\cS(t)\overline{u})(x)$ 
is an entropy admissible solution of \eqref{eq:IVP}. It is the unique entropy admissible solution that is obtained as $L^1$-limit of the wave-front tracking algorithm.
\end{enumerate}
\end{theorem}

\begin{remark}[Uniqueness]
 While \eqref{eq:IVP} may have several entropy admissible solutions there is one and only one entropy admissible solution induced by the SRS; in this sense entropy admissible solutions induced by SRS are unique.
 It was proven in \cite{BressanLeFloch1997} that the SRS-induced entropy admissible solution is the only entropy admissible solution satisfying certain additional stability properties, cf. 
 \cite[(A2),(A3)]{BressanLeFloch1997}.
\end{remark}

\begin{remark}[Domain of SRS]
 The domain of the SRS is discussed in  \cite[equation (1.3)]{BressanLeFloch1997}.
  Note that it can always be replaced by a smaller set in order to make sure that additional
   conditions (such as 
 \eqref{def:domainSemigroup}) hold.
 \end{remark}

Additionally, we will use the following 
result on the stability of the SRS. In particular, we can quantify 
how much the SRS-induced entropy admissible solution varies if the flux is changed.
\begin{theorem}[\hspace{-1pt}\cite{BianchiniColombo2002},  Corollary 2.5]\label{thm:stabilityFlux}
Let the flux function $\Flux$ satisfy  \eqref{ass:flux}
and assume
\begin{align} \label{def:domainSemigroup}
\cD(\Flux) \subseteq \{u \in L^1(\R,\C)~|~\TV(u) \leq M\},
\end{align}
for some suitable positive $M\in \R$ and some compact set $\C\subset \R^m$. For $t>0$ we denote by $\cS(t,\Flux)$ the SRS from \thmref{thm:entropySol}, associated with 
the flux function $\Flux$.

Then there exists a constant $C>0$, such that for any flux function
$ \tilde{\Flux}$, satisfying \eqref{ass:flux} and $\cD(\Tilde{\Flux}) \subseteq \cD(\Flux)$, it holds that
\begin{align}
\|\cS(t,\Flux)u- \cS(t,\tilde{\Flux})u\|_{L^1(\R,\R^m)} \leq C t \|\D \Flux -\D\tilde{\Flux}\|_{C^0(\U, \R^m)},
\end{align}
for all $u\in \cD(\Tilde{\Flux})$.
\end{theorem}

\let\oldemptyset\emptyset
\subsection{Existence and Uniqueness of Random Entropy Solutions} \label{subsec:existenceanduniquenes}
We now consider the random hyperbolic conservation law \eqref{eq:RIVP} and introduce the notion of a random entropy admissible solution for \eqref{eq:RIVP}.  
We say that $\eta\in \leb{1}{\Xi}{C^2(\U;\R^m)}$, \ $q \in \leb{1}{\Xi}{C^2(\U;\R^m)}$ form a random entropy/entropy-flux pair
if $\eta(\cdot,y)$ is strictly convex $\Pas$
and if $\eta$ and $q$ satisfy $\D \eta(\cdot,y) \D \Flux(\cdot,y) =\D q(\cdot,y)$, $\Pas$.
We assume that the random conservation law \eqref{eq:RIVP} is equipped with at least one random entropy/entropy-flux pair.

We define random entropy admissible solutions as path-wise (w.r.t. $y$) 
entropy admissible solutions of \eqref{eq:RIVP}. In this sense the notion of random entropy admissible solutions generalizes the notion of entropy admissible solutions 
in a similar way as the notion of random entropy solutions, 
introduced by Schwab and Mishra \cite{MishraSchwab2012} generalizes the notion of entropy solutions according to Kru\v{z}kov \cite{kruvzkov1970}.


\begin{definition}[Random entropy admissible solution]
A function $\stochSol \in \leb{1}{\Xi}{L^1((0,T)\times \R;\U)}$ is called a random entropy admissible
solution of \emph{\eqref{eq:RIVP}}, if it satisfies the following conditions:
\begin{enumerate}
\item It is a weak solution:
\begin{align} \label{eq:randomentropySol}
\int \limits_0^T \int \limits_{\R} \Big( \stochSol(t,x,y) \cdot \partial_t \phi(t,x) & + \Flux(\stochSol(t,x,y),y) \cdot \partial_x\phi(t,x)\Big)~\mathrm{d}x \mathrm{d}t \notag
\\  &=  - \int \limits_{\R} \stochSol^0(x,y) \cdot \phi(0,x)~ \mathrm{d}x,
\end{align}
$\Pas$ and for all $\phi\in C_c^\infty([0,T)\times \R;\R^m)$.
\item It satisfies the weak entropy inequality:
\begin{align} \label{ineq:randomentropySol}
\int \limits_0^T \int \limits_{\R} \Big( \eta(\stochSol(t,x,y),y) \partial_t \Phi(t,x)  & +q(\stochSol(t,x,y),y)  \partial_x\Phi(t,x)\Big)~\mathrm{d}x \mathrm{d}t \notag
\\ & ~\geq- \int \limits_{\R} \eta(\stochSol^0(x,y),y) \Phi(0,x)~ \mathrm{d}x , 
\end{align}
$\Pas$ and for all $\Phi\in C_c^\infty([0,T)\times \R;\R_+)$.
\end{enumerate}
\end{definition}

\begin{remark}\label{rem:randomentropSol}
\begin{enumerate}[label=(\roman*)]
 \item Let $u$ be a random entropy admissible solution, then for almost any fixed realization $\tilde{y}\in \Xi$, the function $u(\cdot,\cdot,\tilde{y})$ is an entropy admissible solution of the deterministic version of  \emph{\eqref{eq:RIVP}}
in the sense of \secref{subsec:detCL}.
\item The function $\stochSol(t,x,y) := \big(\cS(t,\Flux(\cdot,y)) u^0(\cdot,y)\big)(x)$, 
where $\{\cS(t,\Flux(\cdot,y))\}_{t\geq0}$ is the SRS from \thmref{thm:entropySol} associated with the flux-function $\Flux(\cdot,y)$, obviously satisfies \eqref{eq:randomentropySol} and \eqref{ineq:randomentropySol}
$\Pas$. Theorem \ref{thm:randomEntropySol} below discusses the regularity of $u$ w.r.t. $y$,
i.e. under which conditions $u$ is indeed a random entropy admissible solution of \eqref{eq:RIVP}.
\end{enumerate}

\end{remark}

To ensure the existence of a random entropy admissible solution of \eqref{eq:RIVP} by applying \thmref{thm:entropySol} and \thmref{thm:stabilityFlux} 
path-wise in $\Xi$ we make the following assumptions:
\begin{enumerate}[label=(R\arabic*),ref=R\arabic*]
 \setcounter{enumi}{2}
  \item We define $\cD := \bigcap_{y\in \Xi} \cD(\Flux(\cdot,y))$, where $\cD(\Flux(\cdot,y))$ is the domain
  of the SRS from \eqref{def:domainSemigroup} in \thmref{thm:stabilityFlux}.
  We assume that $\cD \neq  \text{\O} $ and $\stochSol^0(\cdot,y)\in \cD$, $\Pas$. \label{ass:randominitialdatainD}
    \item There exists a compact and convex set $\C\subset \U$, s.t.
  $\cS(t,\Flux(\cdot,y)) u^0(\cdot,y)(x) \in \C$, a.e. $(t,x,y)\in (0,T)\times \R\times ~\Xi$ and $u^0(x,y) \in \C$,  a.e. $(x,y) \in \R \times ~\Xi$. \label{ass:randomcompactSet}
\end{enumerate}


\begin{theorem} \label{thm:randomEntropySol}
Let the assumptions \eqref{ass:randominitialdata}-\eqref{ass:randomcompactSet} hold. 
For $\Pas$ we define $\stochSol(t,x,y) := \cS(t,\Flux(\cdot,y)) u^0(\cdot,y)(x)$, 
where $\{\cS(t,\Flux(\cdot,y))\}_{t\geq0}$ is the SRS from \thmref{thm:entropySol} associated with the flux-function $\Flux(\cdot,y)$. \\
Then $\stochSol$ is a random entropy admissible solution of \eqref{eq:RIVP}. 
It is unique in the sense that it is the only random entropy admissible solution 
which path-wise coincides with the SRS-induced entropy admissible solution of 
the deterministic version of \eqref{eq:RIVP}.
\end{theorem}
\begin{proof}
The function $\stochSol$ is path-wise the unique SRS-induced entropy solution of \eqref{eq:RIVP} by construction.
%
Note that we have assumed $u^0(\cdot,y) \in \cD \subset \cD(\cdot, F(\cdot, y))$, $\Pas$.
It remains to show, that $u$ is a random variable, i.e.  
$$\Big( \Xi, \mathcal{B}(\Xi)\Big) \ni y \mapsto u(\cdot,\cdot,y)\in \Big(\lebnonweight{1}{(0,T)\times \R}{\R^m},  \mathcal{B}((\lebnonweight{1}{(0,T)\times \R}{\R^m}))\Big)$$
 is a measurable map.
To this end we define the vector space
\begin{align*}
E_1:= \lebnonweight{1}{\R}{\R^m} \times C^2(\U;\R^m),
\end{align*}
equipped with the norm
\begin{align*}
\|(\overline{u},\overline{\Flux})\|_{E_1}:= \|\overline{u}\|_{\lebnonweight{1}{\R}{\R^m}} + \|\overline{\Flux}\|_{C^2(\U;\R^m)}.
\end{align*}
Using \thmref{thm:entropySol} (i) and \thmref{thm:stabilityFlux}, which we can apply due to 
assumptions \eqref{ass:randominitialdatainD} and \eqref{ass:randomcompactSet}, we deduce
\begin{align*}
 \|  \cS&(t,\Flux(\cdot,y)) u^0(\cdot,y) - \cS(t,\Flux(\cdot,\tilde{y})) u^0(\cdot,\tilde{y})\|_{\lebnonweight{1}{\R}{\R^m}} \\
 \leq &~\| \cS(t,\Flux(\cdot,y)) u^0(\cdot,y) - \cS(t,\Flux(\cdot,\tilde{y})) u^0(\cdot,{y})\|_{\lebnonweight{1}{\R}{\R^m}}  \\
 &+   \| \cS(t,\Flux(\cdot,\tilde{y})) u^0(\cdot,{y})- \cS(t,\Flux(\cdot,\tilde{y})) u^0(\cdot,\tilde{y})\|_{\lebnonweight{1}{\R}{\R^m}} \\
 \leq & ~C t \|\D \Flux(\cdot,y) -\D\Flux(\cdot,\tilde{y})\|_{C^0(\U, \R^m)} + L \|u^0(\cdot,y)- u^0(\cdot,\tilde{y})\|_{\lebnonweight{1}{\R}{\R^m}} \\
 \leq& ~Ct \|\Flux(\cdot,y)-\Flux(\cdot,\tilde{y})\|_{C^2(\U;\R^m)} + L \|u^0(\cdot,y)- u^0(\cdot,\tilde{y})\|_{\lebnonweight{1}{\R}{\R^m}}, 
\end{align*}
for $\tilde{\P}$-a.s. $y,\tilde{y}\in \Xi$.

Hence, the mapping $\mathrm{S}(t): (\overline{u},\overline\Flux) \ni E_1 \to \lebnonweight{1}{\R}{\R^m}$, $\mathrm{S}(t)(\overline{u},\overline\Flux):=\cS(t,\overline{\Flux})\overline{u}(\cdot)$ is continuous for all $t>0$. 
Due to the finite time horizon we immediately deduce that the mapping 
$$ \mathrm{S}: E_1  \to \lebnonweight{1}{(0,T)\times \R}{\R^m},~\mathrm{S}(\overline{u},\overline\Flux):=\cS(\cdot,\overline{\Flux})\overline{u}(\cdot)$$ is also continuous.
Finally, it follows from assumptions \eqref{ass:randominitialdata} and \eqref{ass:randomflux} that the mapping 
$$   \mathrm{S}_0: \Big(\Xi,\mathcal{B} (\Xi)\Big)  \to  \Big(E_1, \mathcal{B}(E_1)\Big), ~\mathrm{S}_0(y) := (u^0(\cdot, y), \Flux(\cdot,y))   $$
 is measurable. Thus, $u(\cdot,\cdot,y)= \cS(\cdot,\Flux(\cdot,y))u^0(\cdot,y)= \mathrm{S}\circ \mathrm{S}_0(y)$ is a composition of measurable mappings and hence measurable itself. 
\end{proof}

\section{Space-Time Stochastic Discretization and Reconstructions} \label{sec:discretization}
A major goal of this paper is to prove an a posteriori estimate for a large class of numerical 
approximations of \eqref{eq:RIVP}. In particular, we consider numerical schemes that combine 
Stochastic Collocation (SC)
with Runge--Kutta Discontinuous Galerkin (RKDG) schemes.
To this end,  we recapitulate the SC method 
for the (non-intrusive) discretization of the random space  $\randomSpace$ as for example in \cite{abgrallmishra2017 , HesthavenXiu2005}. Additionally, we recall the Multi-Element method
which decomposes the random space $\Xi$ into smaller elements to allow for a local interpolation
in the random space \cite{WanKarniadakis2006a}.
Finally, we describe a reconstruction of the numerical solution as a Lipschitz continuous function.
The reconstruction will be used in the a posteriori error estimate in \secref{sec:aposteriori}.
\subsection{The Stochastic Collocation Method}
The idea of the SC method is to approximate the random entropy admissible solution of \eqref{eq:RIVP} by  a polynomial interpolant in the random space, where the interpolant is supposed to satisfy \eqref{eq:RIVP}
at collocation points $\{y_i\}_{i=0}^M \subset \R$, $M \in \N$. 
The exact solution $\stochSol(\cdot,\cdot, y_i)$ at a given collocation point $y_i$, $i=0,\ldots,M$, is then approximated by a discrete solution $\stochSol_h(\cdot,\cdot, y_i)$ 
using any suitable numerical method. 

For a multi-dimensional random space $\randomSpace \subset \R^N$, we define $\Xi_i:= \xi_i(\Omega)$ and follow \cite{BabuskaNobileTempone2010} to first define the space $\mathcal{P}_q(\randomSpace)$ 
of tensor product polynomials of maximal degree 
 $q \in \N_0$ by
\begin{align*}
\mathcal{P}_q(\randomSpace):= \bigotimes\limits_{i=1}^{N}\mathcal{P}_{q}(\randomSpace_i), \quad \mathcal{P}_{q}(\randomSpace_i):=\{p:\randomSpace_i \to \R~|~p \text{ is a polynomial of degree } q \}.
\end{align*} 

\begin{remark}
Our analysis does not depend on the structure of the approximating space, i.e.~instead of considering a fixed polynomial degree $q \in \N_0$ for every random dimension 
we could also consider  variable polynomial degrees $q_i \in \N$, $i=1,\ldots,N$.
Moreover, instead of using a tensor product space we could also consider the complete polynomial space, cf. \cite{HesthavenXiu2005} or 
sparse grids, cf. \cite{GrieblBungartz2004, NobileTemponeWebster2008}.
\end{remark}

We let $\mathcal{K}:= \{\globalIdx=(k_1,\ldots,k_N)^\top  \in \N_0^{N}: k_i \leq q,~ i=1\ldots,N \}$
be the corresponding multi-index set and define the collocation points $y_\globalIdx=(y_{1,k_1},\ldots,y_{N,k_N})\in \randomSpace$, $\globalIdx\in\cK$. 
As a basis of $\mathcal{P}_{q}(\Xi_i)$ we choose the Lagrange basis $\{l_{i,k} \}_{k=0}^{q}$ associated with the collocation points $\{y_{i,k}\}_{k=0}^{q}$, such that
\begin{align*}
l_{i,k}(y_{i,j})= \delta_{k,j},\qquad \forall ~j,k=0,\ldots,q,
\end{align*}
for all $i=1,\ldots,N$.
We then define the multivariate Lagrange polynomials via
\begin{align*}
l_{\globalIdx}(y_\mathrm{j}):= l_{1,k_1}(y_{1,j_1})\cdot \ldots \cdot l_{N,k_N}(y_{N,j_N}),\quad \mathrm{j},\globalIdx \in \mathcal{K}.
\end{align*}
Using the collocation points $\{y_\globalIdx\}_{\globalIdx\in \cK}$ as input parameters in \eqref{eq:RIVP} yields $\text{card}(\cK)=(q+1)^{N}$ (uncoupled) collocated initial value problems:
\begin{align}\label{eq:DIVP}\tag{CIVP}  
\begin{cases} 
\partial_t \stochSol(t,x,y_\globalIdx)+ \partial_x \Flux(\stochSol(t,x,y_\globalIdx),y_\globalIdx)=0 ,&(t,x) 
\in (0,T)\times \Lambda,
\\ \stochSol(0,x,y_\globalIdx)= \stochSol^0(x,y_\globalIdx), ~ &\Lambda \in \R.
\end{cases}
\end{align}
Here and in the following we consider $\Lambda \in \{[0,1]_{per},\R\}$.
\begin{remark}
The well-posedness result in \thmref{thm:randomEntropySol} only covers $\Lambda = \R$. 
However, the deterministic well-posedness results are based on local estimates and we, therefore,
believe that it can be extended to cover the case $\Lambda = [0,1]_{per}$, as well. 
\end{remark}
Each of the deterministic hyperbolic systems in \eqref{eq:DIVP} can be solved using the
RKDG method described in \appAref{appendix:A}.
For every collocation point $y_\globalIdx$ we denote the corresponding numerical approximation at time $t_n=t_n(y_\globalIdx)$ by $\numSol^n(\cdot,y_\globalIdx):= \numSol(t_n,\cdot, y_\globalIdx)$. Let us assume that the time-partition $\{t_n\}_{n=0}^{\ntCells}$ is the same for every collocation point $\{y_\globalIdx\}_{\globalIdx\in \cK}$.
The numerical approximation of \eqref{eq:RIVP} at time $t=t_n$ can then be written as
\begin{align} \label{def:stochSolSC}
\numSol^n(x,y):= \sum \limits_{\globalIdx\in \cK} \numSol^n(x,y_\globalIdx) l_\globalIdx(y).
\end{align}
An important aspect of the SC method is the choice the collocation points $\{y_{i,k} \}_{k=0}^{q}\subset \Xi_i$. 
Depending on the distribution of the  random variable $\xi_i$ we choose the collocation points as zeros of the corresponding (orthogonal) chaos polynomials \cite{XiuKarniadakis2002}. 
For example, if $\xi_i\sim \mathcal{U}(a,b)$ is uniformly distributed, we choose $\{y_{i,k} \}_{k=0}^{q}$ to be the roots of the $(q+1)$-th Legendre polynomial.
For a Gaussian distribution we use the roots of the Hermite polynomials accordingly.

One can then approximate the mean of $\numSol^n(x,\cdot)$ via numerical quadrature, i.e.
\begin{align*}
\E\Big(\numSol^n(x,\cdot)\Big) = \int \limits_{\randomSpace} \numSol^n(x,y)~p_{\xi}(y)\mathrm{d}y \approx \sum \limits_{\globalIdx\in \cK} \numSol^n(x,y_\globalIdx) w_\globalIdx.
\end{align*}
Here $w_\globalIdx$ are products of the corresponding one-dimensional weights.

\subsection{The Multi-Element Stochastic Collocation Method} \label{subsec:ME}                  
A major drawback of any global approximation  approach in $\Xi$ for hyperbolic conservation laws is that, due to the Gibbs phenomenon, the interpolant may
oscillate for discontinuous solutions, cf. \cite{WanKarniadakis2005, PoetteDespresLucor2009}. To overcome this issue, we employ the Multi-Element (ME) approach as presented in \cite{WanKarniadakis2006a}, i.e. we subdivide $\Xi$ into disjoint elements and consider a local approximation of  \eqref{eq:RIVP} on every random element.

For the ease of presentation we assume that $\Xi=[0,1]^N,$ and let $0= d_1 < d_2 < \ldots < d_{\MEElements+1}= 1$ be a decomposition of $[0,1]$.
We define $\randomElement{n} := [d_{n}, d_{n+1} )$, for $ n= 1,\ldots, \MEElements-1,$ and 
$\randomElement{\MEElements} := [d_{\MEElements}, d_{\MEElements+1}]$.
Introducing the set $\mathcal{M}:= \{\MEIndex=(m_1,\ldots,m_N)^\top  \in \N_0^{N}: m_i \leq \MEElements,~ i=1\ldots,N \}
$ allows us to define for $\MEIndex\in \mathcal{M}$, the Multi-Element $\randomElement{\MEIndex} := \bigtimes \limits_{i=1}^N \randomElement{m_i}$.
Hence, we consider a new local random variable $\xi^{\MEIndex}: \xi^{-1}(D_{\MEIndex}) \to \randomElement{\MEIndex}$ 
on the local probability space $(\xi^{-1}(D_{\MEIndex}),\F\cap \xi^{-1}(D_{\MEIndex}),\P(\cdot | \xi^{-1}(D_{\MEIndex})))$.
Using Bayes' rule we can compute the local probability 
density function of $\xi^{\MEIndex}$ via
\begin{equation}\label{def:CondDensity}
w_{\xi^{\MEIndex}}:=w_{\xi}(y^\MEIndex|\xi^{-1}(D_{\MEIndex})) = \frac{w_\xi(y^\MEIndex)}{\P(\xi^{-1}(D_{\MEIndex}))} ,\qquad y^\MEIndex \in D_\MEIndex,
\end{equation}
where $\P(\xi^{-1}(D_{\MEIndex}))>0$ for $\MEIndex \in \mathcal{M}$ can be assumed w.l.o.g., due to the independence of the corresponding random variables.
We parametrize the uncertain input in \eqref{eq:RIVP}  using the local random variable $\xi^{\MEIndex}$ and consider a local approximation on every $\randomElement{\MEIndex}$
at time $t=t_n$, $n=0,\ldots,\ntCells$,
\begin{equation}
\numSol^{\MEIndex}(t_n,x,y) = \sum_{\globalIdx\in\cK} \numSol(t_n,x, y_\globalIdx^\MEIndex) l_{\globalIdx}^{\MEIndex} (y),
\end{equation}
for all $\MEIndex  \in \mathcal{M}$. Here, $\{y_{\globalIdx}^{\MEIndex} \}_{\globalIdx\in \cK} \subset D_\MEIndex$ are the local collocation points in $D_\MEIndex$, $\MEIndex \in \mathcal{M}$.
The global Multi-Element Stochastic Collocation (ME-SC) approximation at time $t=t_n$, can then be written as
\begin{align}\label{def:globalSG}
\numSol(t_n,x,y)= \sum_{ \MEIndex \in \mathcal{M}}  \numSol^\MEIndex(t_n,x,y) \chi_{\D_\MEIndex}(y) = \sum_{ \MEIndex \in \mathcal{M}}  \sum_{\globalIdx\in\cK} \numSol(t_n,x, y_\globalIdx^\MEIndex) l_{\globalIdx}^{\MEIndex} (y) \chi_{D_\MEIndex}(y),
\end{align}
where $\chi_{D_\MEIndex}$ is the indicator function of $D_\MEIndex$.
\subsection{Space-Time-Stochastic Reconstructions}\label{subsec:DG}
For the space-time discretization of \eqref{eq:DIVP} we use the  RKDG framework from \cite{cockburn1998}. 
For ease of presentation, we move the description of the RKDG scheme and the computation of its 
space-time reconstruction to \appAref{appendix:A}.

As discussed in \appAref{appendix:A}, we have for each collocation point $y_{\globalIdx}$, $\globalIdx \in \cK$,
a computable space-time reconstruction $\reconst{st}(y_{\globalIdx}) = \reconst{st}(\cdot,\cdot, y_{\globalIdx})  \in W_\infty^1((0,T);V_{p+1}^s\cap C^0(\Lambda))$  of the
numerical approximation $\numSol(\cdot,\cdot,y_{\globalIdx})$, where $V_{p+1}^s$ denotes the space of piece-wise polynomials of degree $p+1$ on a triangulation of $\Lambda$.
This allows us to define the space-time residual as follows.
\begin{definition}[Space-time residual]
We call the function $\residual{st}(y_\globalIdx) :=\residual{st}(\cdot,\cdot,y_\globalIdx) \in \LpSpace{2}{(0,T)\times \Lambda}{\R^m}$, defined by
\begin{align}\label{eq:STResidualColl}
\residual{st}(t,x, y_\globalIdx):= \partial_t \reconst{st}(t,x, y_\globalIdx) + \partial_x \Flux(\reconst{st}(t,x, y_\globalIdx),y_\globalIdx),
\end{align}
the space-time residual associated with the collocation point $y_\globalIdx$, for all~$\globalIdx\in\cK$.
\end{definition} 
This residual is required in the subsequent analysis.
In the next step we expand the space-time reconstruction into the corresponding random basis, i.e. in the Lagrange basis, 
to obtain the so-called space-time-stochastic reconstruction.

\begin{definition}[Space-time-stochastic reconstruction]\label{def:STSReconst}
We call the function $\reconst{sts} \in \mathcal{P}_q(\randomSpace)\otimes \big( W_\infty^1(0,T); V_{p+1}^s \cap C^0(\Lambda)\big)$ defined by
$$\reconst{sts}(t,x,y):= \sum \limits_{\globalIdx\in \cK} \reconst{st}(t,x,y_\globalIdx)l_\globalIdx(y),$$
the 
space-time-stochastic reconstruction of the numerical approximation $\numSol$ of \eqref{eq:RIVP} (see \eqref{def:stochSolSC}).
\end{definition} 
Similar to the space-time reconstruction, we may plug $\reconst{sts}$ into the random conservation law \eqref{eq:RIVP} to obtain the so called
space-time-stochastic residual.
\begin{definition}[Space-time-stochastic residual] \label{def:STSResidual}
We define the space-time-stochastic residual $\residual{sts} \in \leb{2}{\Xi}{\Leb^2((0,T)\times \Lambda;\R^m)}$ by
\begin{align}\label{eq:STSResidual}
\residual{sts}(t,x,y):= \deriv{t} \reconst{sts}(t,x,y) + \deriv{x}\Flux(\reconst{sts}(t,x,y), y) .
\end{align}
\end{definition}
We also need this residual for the upcoming error analysis.
\section{A Posteriori Error Estimate and Adaptive Algorithms }\label{sec:aposteriori}
As already mentioned in the introduction, our a posteriori error estimate relies on the relative entropy stability framework of Dafermos and DiPerna, see \cite{Dafermos2016} and references therein. 
The relative entropy method allows to measure the $L^2$-distance of two
functions, one of them required to be Lipschitz continuous in space and time. This is the reason why we reconstructed the numerical solution and computed the space-time-stochastic reconstruction
as a Lipschitz function.

Before stating the main a posteriori error estimate, we establish bounds on the derivatives of the flux function and the entropy,
as they enter the upper bounds in the main estimate.
Due to \assref{ass:randomcompactSet} from \secref{subsec:existenceanduniquenes}
and the compactness of $\C$, there  exist $\Pas$ constants $0<\fluxConstant(y) < \infty$ and $0<\lowerEta(y) < \upperEta(y)< \infty$, such that,
\begin{align*}
|v^\top H_u \Flux(u,y)v|\leq \fluxConstant(y)|v|^2, \quad \lowerEta(y)|v|^2 \leq v^\top H_u \eta(u,y)v \leq \upperEta(y) |v|^2, \\ \forall v\in \R^m, \forall u\in \C.
\end{align*}
Here, for a generic function $f$, $H_uf$ denotes its Hessian matrix which contains all second order derivatives
with respect to $u$.
\subsection{A Posteriori Error Estimate and Error Splitting}
 
We now have all ingredients together to state the following main a posteriori error estimate
that can be directly derived from Theorem 5.5 in \cite{GiesselmannMakridakisPryer15}.
\begin{theorem}[A posteriori error bound for the numerical solution] 
\label{thm:stochApostNumSol}
Let $u$ be a random entropy admissible solution of \eqref{eq:RIVP}. Let the reconstruction $\reconst{sts}$ only take values in $\mathcal{C}$. Then, the difference between $u$ and the numerical solution $\numSol^n$ from  \eqref{def:stochSolSC} satisfies
\begin{align*} 
  \|\stochSol&(t_n,\cdot,\cdot)-\numSol^n(\cdot,\cdot)\|_{\leb{2}{\tilde{\Xi} }{\Leb^2(\Lambda)}}^2 \\
   \leq& ~ 2\|\reconst{sts}(t_n,\cdot,\cdot)- \numSol^n(\cdot,\cdot)\|_{\leb{2}{\tilde{\Xi} }{\Leb^2(\Lambda)}}^2 \\
   &+ 
 2 \int \limits_{\tilde{\Xi}} \Big[\Big(\lowerEta^{-1}(y) \Big( \stsResidual(t_n,y) + \upperEta(y) \stsResidual_0(y) \Big)\Big) 
 \\& \times \exp\Big(\int \limits_0^{t_n}  \frac{\upperEta(y) \fluxConstant(y) \|\partial_x \reconst{sts}(t,\cdot,y)\|_{\Leb^\infty(\Lambda)} + \upperEta^2(y)} {\lowerEta(y)}~\mathrm{d}t  \Big) \Big] ~w(y)\mathrm{d}y,
\end{align*}
for all $n=0,\ldots,\ntCells$ and for any $\tilde{\P}$-measurable set $\tilde{\Xi} \subseteq \Xi$.
Here
\begin{align} 
&\stsResidual(t_n,y):= \|\residual{sts}(\cdot,\cdot,y)\|_{\Leb^2((0,t_n)\times \Lambda)}^2, \\
&\stsResidual_0(y):= \|u^0(\cdot,y)-\reconst{sts}(0,\cdot,y)\|_{\Leb^2(\Lambda)}^2.
\end{align}
\end{theorem} 

\begin{proof}
Thanks to the path-wise structure we can apply the 
deterministic setting from \cite[Theorem 5.5 ]{GiesselmannMakridakisPryer15} for almost
any  $y\in  \tilde{\Xi} \subseteq \Xi$. Integrating over $\tilde{\Xi}$ gives the desired
estimate directly. 
\end{proof}

In \thmref{thm:stochApostNumSol} the error between the numerical solution and the random entropy 
admissible solution is bounded
by  the error in the initial condition, the difference between the numerical solution and its reconstruction,
and the contribution of the space-time stochastic residual $\residual{sts}$ from \defref{def:STSResidual}, quantified by $\stsResidual$.
However, we still cannot distinguish between errors that arise from discretizing the random space and the physical space. 
We, thus, are going to describe a splitting of the space-time-stochastic residual $\residual{sts}$ into a spatial and a stochastic residual. 
\begin{lemma}[Splitting of the space-time-stochastic residual] \label{lemma:splitting}
The space-time-stochastic residual $\residual{sts}$ admits the decomposition
\begin{align} \label{eq:splittingSC}
\residual{sts} = \residual{det} + \residual{stoch}, 
\end{align}
with 
\begin{align}
\residual{det}& :=    \sum \limits_{\globalIdx\in \cK}\residual{st}(y_\globalIdx)l_\globalIdx \text{ and } \label{def:detResidual}
\\  \residual{stoch} &:=  \Big(\deriv{x} \Flux\Big( \sum \limits_{\globalIdx\in \cK} \reconst{st}(y_\globalIdx,\cdot) l_\globalIdx(\cdot)\Big) 
- 
\sum \limits_{\globalIdx\in \cK} \deriv{x} \Flux(\reconst{st}(y_\globalIdx),y_\globalIdx) l_\globalIdx   (\cdot)\Big) . \label{def:stochResidual}
\end{align}
$\residual{det}$ and $\residual{stoch}$
are called the deterministic and stochastic residual.
\end{lemma}
\begin{proof}
For every collocation point $y_\globalIdx,$ $\globalIdx\in\cK$, we compute in \appAref{appendix:A} the space-time reconstruction 
$\reconst{st}(\cdot,\cdot,y_\globalIdx)$ which fulfills
\begin{align} \label{eq:determinsticResidual} 
\residual{st}(y_\globalIdx)= \deriv{t} \reconst{st}(y_\globalIdx) + \deriv{x} \Flux(\reconst{st}(y_\globalIdx),y_\globalIdx ).
\end{align}
Moreover, we know from \eqref{eq:STSResidual}  that the space-time-stochastic reconstruction
 $\reconst{sts}=\sum \limits_{\globalIdx\in \cK} \reconst{st}(t,x,y_\globalIdx)l_\globalIdx(y)$
satisfies the relation
\begin{align} \label{eq:spacestochResidual}
\residual{sts}= \deriv{t} \reconst{sts} + \deriv{x}\Flux(\reconst{sts},\cdot)
= \deriv{t}\Big(\sum \limits_{\globalIdx\in \cK} \reconst{st}(y_\globalIdx)l_\globalIdx \Big) + \deriv{x} \Flux \Big(\sum \limits_{\globalIdx\in \cK} \reconst{st}(y_\globalIdx) l_\globalIdx, \cdot\Big).
\end{align}
Multiplying \eqref{eq:determinsticResidual} by $l_\globalIdx$ and summing over $\globalIdx \in \cK$ yields
\begin{align}\label{eq:sumdeterministicResidual}
\sum \limits_{\globalIdx\in \cK}\residual{st}(y_\globalIdx)l_\globalIdx
= \sum \limits_{\globalIdx\in \cK} \deriv{t} \reconst{st}(y_\globalIdx) l_\globalIdx+ \sum \limits_{\globalIdx\in \cK} \deriv{x} \Flux(\reconst{st}(y_\globalIdx),y_\globalIdx) l_\globalIdx. 
\end{align}
Inserting (\ref{eq:sumdeterministicResidual}) into (\ref{eq:spacestochResidual}) yields
\begin{align*}
\residual{sts}=& ~\deriv{t}\Big(\sum \limits_{\globalIdx\in \cK} \reconst{st}(y_\globalIdx) l_\globalIdx \Big) +
 \deriv{x} \Flux\Big( \sum \limits_{\globalIdx\in \cK} \reconst{st}(y_\globalIdx) l_\globalIdx, \cdot\Big) \notag
\\ 
&+  \sum \limits_{\globalIdx\in \cK}\residual{st}(y_\globalIdx)l_\globalIdx -
 \Big(\sum \limits_{\globalIdx\in \cK} \deriv{t} \reconst{st}(y_\globalIdx) l_\globalIdx+ \sum \limits_{\globalIdx\in \cK} \deriv{x} \Flux(\reconst{st}(y_\globalIdx),y_\globalIdx) l_\globalIdx  \Big) \notag
\\ 
=& \sum \limits_{\globalIdx\in \cK}\residual{st}(y_\globalIdx)l_\globalIdx+ 
\Big(\deriv{x} \Flux\Big( \sum \limits_{\globalIdx\in \cK} \reconst{st}(y_\globalIdx) l_\globalIdx,\cdot\Big)
- 
\sum \limits_{\globalIdx\in \cK} \deriv{x} \Flux(\reconst{st}(y_\globalIdx),y_\globalIdx) l_\globalIdx   \Big)  \\
=&  \residual{det} + \residual{stoch}.  
\end{align*}
\end{proof}

To simplify \thmref{thm:stochApostNumSol} let us assume that the eigenvalues 
of the Hessian $H_u\eta(u,y)$  are bounded from above
and below by positive numbers, for any $u \in \C$ uniformly in $\Xi$. 
We let $\upperEta:= \operatorname{ess}\sup \limits_{y \in \Xi} \upperEta(y) < \infty$,
$\lowerEta:= \operatorname{ess}\inf \limits_{y \in \Xi} \upperEta(y) >0$
and $\fluxConstant:= \operatorname{ess}\sup \limits_{y \in \Xi} \fluxConstant(y) < \infty$. 
In our numerical examples, where we consider the compressible Euler equations, 
the dependence of the flux function $\Flux$ and $\eta$ on $y$ is explicitly known 
and we can compute the constants $\lowerEta, \upperEta, \fluxConstant$
numerically.

The following corollary is a simple consequence of the splitting in \lemref{lemma:splitting}.
\begin{corollary}[A posteriori error bound with error splitting and simplified bounds] \label{cor:stochApostEasy}
Let $u$ be a random entropy admissible solution of \emph{\eqref{eq:RIVP}}. Then, the difference between $u$ and the numerical solution $
\numSol^n$ from \eqref{def:stochSolSC} satisfies
\begin{align*}
\|\stochSol(t_n,\cdot,\cdot)-\numSol^n(\cdot,\cdot)\|_{\leb{2}{\tilde{\Xi}}{\Leb^2(\Lambda)}}^2   
\leq & 
 ~2\|\reconst{sts}(t_n,\cdot,\cdot)- \numSol^n(\cdot,\cdot)\|_{\leb{2}{\tilde{\Xi} }{\Leb^2(\Lambda)}}^2 
 \\ & + 2 \lowerEta^{-1} \Big( 2\detResidual(t_n)+2\stochResidual(t_n) + \upperEta \stsResidual_0  \Big) 
 \\ & \times   \exp\Big( \lowerEta^{-1} \int \limits_0^{t_n}   \Big(\upperEta \fluxConstant \|\partial_x \reconst{sts}(t,\cdot,y)\|_{\leb{\infty}{\tilde{\Xi}}{\Leb^\infty(\Lambda)}} 
 + \upperEta^2\Big)  ~\mathrm{d}t \Big) 
\end{align*}
for $n=0,\ldots, \ntCells$ and for all $\tilde{\P}$-measurable sets $\tilde{\Xi} \subseteq \Xi$.
Here,
 \begin{align} 
 \detResidual(t_n)     &:= \|\residual{det}\|_{\leb{2}{\tilde{\Xi}}{\Leb^2((0,t_n)\times \Lambda}}^2, \label{def:spatialIndicator}\\
 \stochResidual(t_n)  &:= \|\residual{stoch}\|_{\leb{2}{\tilde{\Xi}}{\Leb^2((0,t_n)\times \Lambda}}^2,  \label{def:stochasticIndicator} \\
 \stsResidual_0       & := \|u^0(\cdot,\cdot)-\reconst{sts}(0,\cdot,\cdot)\|_{\leb{2}{\tilde{\Xi}}{
 	 \Lambda}}^2.
 \end{align}
\end{corollary}

\begin{remark} 
\begin{enumerate}[label=(\roman*)]
\item The residual $\residual{det}$ in \eqref{def:detResidual} interpolates spatio-temporal residuals
and contains information about the discretization error in physical space, 
i.e. the space-time resolution of \eqref{eq:DIVP} using the RKDG method.
In contrast to  $\residual{det}$,
the stochastic residual $\residual{stoch}$ in \eqref{def:stochResidual} 
indicates the quality of the interpolation in stochastic space.
\item In order for the upcoming space-stochastic adaptive algorithm based on
$\detResidual$, $\stochResidual$ to be efficient, we need $\detResidual$ to depend solely on the
space-time discretization and to be
independent of the stochastic discretization. Similarly, we need $\stochResidual$ to
decay when the stochastic resolution is increased but 
to be independent of the space-time discretization.
\\ In \rmref{rm:uniformity} we prove that $\detResidual$ is indeed 
unaffected by stochastic refinement.
Our numerical experiment in \secref{subsubsec:smooth1D}
also shows that $\stochResidual$ is unaffected by spatial refinement.
\item The scaling properties of $\detResidual$ , resp. $\residual{st}(y_\globalIdx)$, were studied in 
\cite{GiesselmannDedner16}. 
Currently we are not able to prove any of the scaling properties of $\stochResidual$ w.r.t. 
to $q$ and the number of Multi-Elements. However, our numerical experiments show 
that $\stochResidual$ scales as desired, i.e. $\stochResidual$ shows the same qualitative behavior as
the stochastic interpolation error of the exact solution.
\item As described in   \cite{GiesselmannDedner16} and \cite{GiesselmannMeyerRohde18},
 $\residual{det}$ scales with $\frac{1}{h}$ in the vicinity of shocks
and contact discontinuities, i.e., it blows up under spatial mesh refinement in these areas, 
although the numerical solution converges towards the exact solution.
Hence, we only have reliable a posteriori error control for smooth solutions of 
\eqref{eq:RIVP}. However, as 
$\residual{det}$  precisely captures the positions of rarefaction waves, contact discontinuities 
and shocks we use $\residual{det}$ and $\residual{stoch}$, resp. $\detResidual$ 
and $\stochResidual$, as local indicators for our adaptive mesh refinement algorithms
described in \secref{subsec:adaptivealgo}. 
\end{enumerate}
\label{rm:residuals}
\end{remark}
\begin{remark}[Uniformity of the deterministic residual in $\Xi$] \label{rm:uniformity}
As noted above, the collocation points $y_\globalIdx$ are chosen to be the zeros of the 
corresponding orthogonal polynomial depending on the distribution of $\xi$.
The deterministic residual $\residual{det}$ from \eqref{def:detResidual} consists of  Lagrange polynomials associated with the corresponding
collocation points, thus Gaussian quadrature in $\Xi$ yields
\begin{align*}
\detResidual(T) = \| \residual{det} \|_{{\leb{2}{\Xi}{\Leb^2((0,T)\times \Lambda}}} &= \sum \limits_{\globalIdx\in \cK} \|\residual{st}(y_\globalIdx)\|_{L^2((0,T)\times \Lambda)} w_\globalIdx 
\leq \max  \limits_{\globalIdx \in \cK} \|\residual{st}(y_\globalIdx)\|_{L^2((0,T)\times \Lambda)}.
\end{align*}
Hence, $\detResidual$ inherits the convergence order of $\residual{st}(y_\globalIdx)$.
\end{remark}

\subsection{Adaptive Algorithms} \label{subsec:adaptivealgo}
The splitting of the space-time-stochastic residual into a deterministic and a stochastic residual helps
us in developing adaptive numerical schemes where we use the residuals as local error indicators for spatial and stochastic mesh refinement. 
We describe the deterministic spatially adaptive algorithm, which we use to solve \eqref{eq:DIVP} for every collocation point 
$y_{\globalIdx}$, $\globalIdx \in \cK$.
We slightly abuse the notation from \eqref{def:spatialIndicator} 
and write for every physical cell $I \in \Tn$,  $\detResidual_{\globalIdx}(t_{n},t_{n+1},I):= \| \residual{st}(y_{\globalIdx})\|_{\Leb^2((t_n,t_{n+1})\times I)}$, 
which is the cell-wise indicator for the spatial refinement in $\Lambda$.

The local physical mesh refinement is achieved by uniformly dividing one cell into two
new children cells or merging two cells into one parent cell. To mark elements for refinement
we compute the deterministic residual  
$\detResidual_{\globalIdx}(t_{n},t_{n+1},I)$ 
on every cell $I \in \Tn$
and based on the residual we mark a fixed fraction of the cells for refinement.
To coarsen the mesh, we can
only merge cells that have the same parent element and both siblings are marked for coarsening.
For coarsening we also choose a fixed fraction of all elements
according to the local residual  $\detResidual_{\globalIdx}(t_{n},t_{n+1},I)$, cf. \cite{HartmannHouston2002}.
Additionally, each cell is augmented with a variable denoting its
current mesh-level which is initially zero. We fix a maximum mesh-level $L\in \N$,
to restrict the fineness of the adaptive mesh. The algorithm reads as follows:

 \begin{algorithm}[H] 
 \caption{Deterministic $h$-adaptive algorithm}
 \textbf{Input:} final time $T$, max mesh-level $L$, initial mesh $\T_{0}$ 
 \begin{algorithmic}[1]   
\State Compute $\numSol^{n+1}$ on the current mesh $\Tn$ using \algoref{algo:TVBMRungeKutta} (see \appAref{appendix:A})
\State Compute $\detResidual_{\globalIdx}(t_{n},t_{n+1},I)$ for $I \in \Tn$ and mark a fixed fraction of the elements
for refinement and coarsening
\begin{algsubstates}
        \State Refinement: If the cell's mesh-level is $L$ do nothing. Else divide it uniformly into two new cells
                       and increase the two new cells' mesh-level  by one
        \State Coarsening: If the cell's mesh-level is zero do nothing. Else check if its sibling is marked for coarsening.
               If yes merge the two cells into one and decrease its mesh-level by one 
      \end{algsubstates}
\State Project $\numSol^{n+1}$ onto the new mesh $\T_{n+1}$ using the $L^2$-projection
\State If $t_{n+1}<T$ go to step 1
\end{algorithmic}
\label{algo:hAdaptive}
\end{algorithm}
\begin{remark}
After every projection step in line three of \algoref{algo:hAdaptive} we apply
the TVBM slope limiter  $\Lambda \Pi_h$ from \appAref{appendix:A}.
\end{remark}
In the numerical experiments in \secref{sec:numerics} we observed that setting the refinement and coarsening fractions to 1\% and 20\% provided the best error reduction.
The finest mesh level will be $L=3$. 
Next, we describe the stochastic adaptive algorithm for the ME-SC method from \secref{subsec:ME} using
the stochastic residual $\stochResidual$ as local indicator for stochastic refinement.

 \begin{algorithm}[H] 
 \caption{Stochastic $\MEElements$-Adaptive Algorithm}
 \textbf{Input:} initial number of Multi-Elements $M_{\Xi}$, max no. of Multi-Elements $\MEElements$, $q+1$ number of collocation points in each stochastic dimension 
 \begin{algorithmic}[1]   
\State For every Multi-Element $D_{\MEIndex}$ compute $(q+1)^N$ numerical samples  using \algoref{algo:hAdaptive}
\State Compute $\stochResidual(T)$ on every Multi-Element $D_{\MEIndex}$ and uniformly subdivide the
  Multi-Element with the biggest residual, set $M_{\Xi}:= M_{\Xi} +(2^N-1)$ 
\State If $M_{\Xi}<\MEElements$ compute $M$ samples on every new Multi-Element and go to 2
\end{algorithmic}
\label{algo:stochAdaptive}
\end{algorithm}

\section{Numerical Examples} \label{sec:numerics}
In this section we present various numerical examples concerning the scaling properties of the residuals and the performance of the adaptive algorithms.
In \secref{subsec:detSmooth} and \secref{subsec:SmoothStoch} we examine the scaling properties of $\detResidual$ and $\stochResidual$.
Sections \ref{subsec:SodShockdet}, \ref{subsec:DiscStoch} and \ref{subsec:spaceStochSod} assess the efficiency of
our proposed adaptive algorithms.

As numerical solver we use the RKDG Code Flexi \cite{Hindenlang2012}.  
The DG polynomial degrees
will always be one or two and for the time-stepping we use the low storage SSP RK-method of order 
three as in \cite{Ketcheson2008}. The time-reconstruction is  also of order three.
As numerical fluxes we choose either the Lax-Wendroff numerical flux
\begin{equation} \label{eq:LW}
  G(u,v):= F(w(u,v)),\quad w(u,v):= \frac{1}{2}\Big((u+v)+ \frac{\Delta t}{h}(F(v)-F(u))\Big),
\end{equation}
or the Lax-Friedrichs numerical flux
\begin{equation} \label{eq:LF}
  G(u,v) := \frac{1}{2}\Big(F(u)+F(v) \Big) + \lambda(v-u).
\end{equation}
In our example, the uncertainty is uniformly distributed. Therefore, we use the zeros of the
Gau\ss-Legendre polynomials as collocation points.
Computing $\detResidual, \stochResidual$ requires computing integrals, we approximate them via
Gau\ss-Legendre quadrature where we use seven points in time, ten points in physical space 
and ten points in random space, except for Example \ref{subsec:SmoothStoch},
where for the global interpolation the number of quadrature points in random space will be $2q$.

In the following experiments we consider as instance of \eqref{eq:RIVP} the one-dimensional compressible Euler equations for the flow of an ideal gas, which
are given by
\begin{equation} \label{eq:euler}
\begin{alignedat}{2}
\hspace*{2cm}
\dt \rho &+ \dx m &&= 0 ,\\
\dt m &+ \dx \left(\frac{m^2}{\rho} + p\right)  &&=0,\\
\dt E &+ \dx \left( (E + p) \,\frac{m}{\rho}\right) &&  = 0,\hspace*{2cm}
\end{alignedat}
\end{equation}
where $\rho$ describes the mass density, $m$ the momentum and $E$ the energy of the gas. 
The constitutive law for pressure $p$ reads
$$p = (\gamma-1)\left(E - \frac{1}{2}\frac{m^2}{\rho}\right),$$
with the adiabatic constant $\gamma=1.4$ if not specified otherwise. 
In the following figures we refer to the quantity $\|m(T,\cdot,\cdot)-m_h^{\ntCells}(\cdot,\cdot)\|_{\leb{2}{\Xi}{\Leb^2(\Lambda)}}$ at final computational time $T$ as numerical error,
 unless otherwise stated.
We also plot the residuals  $\detResidual(T)$ and $\stochResidual(T)$ as in \eqref{def:spatialIndicator} and \eqref{def:stochasticIndicator} from the momentum equation. 
\begin{remark}
  Due to the structure  of the flux Jacobian for the Euler equations \eqref{eq:euler},
  \begin{align*} \D F(u)=
    \begin{pmatrix}
     0 & 1 & 0 \\
     -0.5(\gamma-3)\frac{m^2}{\rho^2} & (3-\gamma)\frac{m}{\rho} & \gamma-1 \\
     -\gamma \frac{Em}{\rho}+ (\gamma-1)\frac{m^3}{\rho^3} & \gamma \frac{E}{\rho}- \frac{3}{2}(\gamma-1)\frac{m^2}{\rho^2} & \gamma \frac{m}{\rho}
    \end{pmatrix},
  \end{align*}
  the first component of the stochastic residual $\residual{stoch}$ from \eqref{def:stochResidual} vanishes when considering the Euler equations  without source term. 
  We therefore use
  the residuals for the momentum and the energy balance as indicators for our space-stochastic mesh refinements.
\end{remark}
\subsection{A Deterministic Problem with Smooth Solution} \label{subsec:detSmooth}
In this numerical example, we study the scaling properties of the deterministic residual
$\detResidual$ from \eqref{def:spatialIndicator} for a uniform spatial
mesh refinement. To this end,
we construct a smooth exact solution by introducing 
an additional source term into the Euler equations. 
The exact solution reads as follows
\begin{align} \label{eq:detbenchmark}
	\begin{pmatrix}
	\rho(t,x,y) \\
	 m(t,x,y)  \\
	E(t,x,y)  
	\end{pmatrix}=
	\begin{pmatrix}
	2+ 0.1\cos(4\pi(x-t)) \\[0.1cm]
	\Big(2+ 0.1\cos(4\pi(x-t)) \Big)\Big(1+0.1\sin(4\pi(x-t))\Big) \\
	\Big(2+ 0.1\cos(4\pi(x-t)) \Big)^2
	\end{pmatrix}.
\end{align}
The numerical solution is computed on the domain $\Lambda=[0,1]_{\mathrm{per}}$ up to $T=0.5$ and we use the Lax-Wendroff numerical flux \eqref{eq:LW}.


In \tabref{table:detSmooth} we present the numerical error and the residual $\detResidual$ from \eqref{def:spatialIndicator}
for the smooth solution \eqref{eq:detbenchmark} for DG polynomial degrees one and two. We can see that the error and the residual 
converge with the correct order of convergence, which is $p+1$, where $p$ is the DG polynomial degree. 

\begin{table}[htb!]
\centering
\begin{filecontents*}{p1ErrorSmooth.csv}
n,l2err,eocErr,residual,eocResidual
16,1.2771e-02,-,1.1821e-01,-
32,4.5795e-03,1.48,4.5907e-02,1.36
64,1.4193e-03,1.69,1.5947e-02,1.53
128,3.5089e-04,2.02,4.3912e-03,1.86
256,7.9658e-05,2.14,1.3655e-03,1.69
512,1.9625e-05,2.02,3.7220e-04,1.88
1024,4.8984e-06,2.00,1.0804e-04,1.78

\end{filecontents*}
\pgfplotstabletypeset[
    col sep=comma,
    string type,
    every head row/.style={%
        before row={\hline
            \multicolumn{5}{|c|}{$p=1$}  \\
            \hline
        },
        after row=\hline
    },
    every last row/.style={after row=\hline},
    columns/n/.style={column name=$\nCells$, column type={|c}},
    columns/l2err/.style={column name=error, column type={|c}},
    columns/eocErr/.style={column name=eoc, column type={|c|}},
    columns/residual/.style={column name=$\detResidual(T)$, column type={|c}},
    columns/eocResidual/.style={column name=eoc, column type={|c|}},
    ]{p1ErrorSmooth.csv}

\begin{filecontents*}{p2ErrorSmooth.csv}
n,l2err,eocErr,residual,eocResidual
16,2.5172e-04,-,5.3627e-03,-
32,1.606e-05,3.97,5.8275e-04,3.20
64,1.7387e-06,3.21,7.0515e-05,3.05
128,2.1568e-07,3.01,8.8464e-06,2.99
256,2.6614e-08,3.02,1.1316e-06,2.97
512,3.3226e-09,3.00,1.4746e-07,2.94
1024,4.1489e-10,3.00,1.9320e-08,2.93
\end{filecontents*}
\pgfplotstabletypeset[
    col sep=comma,
    string type,
    every head row/.style={%
        before row={\hline
            \multicolumn{5}{|c|}{$p=2$}  \\
            \hline
        },
        after row=\hline
    },
    every last row/.style={after row=\hline},
    columns/n/.style={column name=$\nCells$, column type={|c}},
    columns/l2err/.style={column name=error, column type={|c}},
    columns/eocErr/.style={column name=eoc, column type={|c|}},
    columns/residual/.style={column name=$\detResidual(T)$, column type={|c}},
    columns/eocResidual/.style={column name=eoc, column type={|c|}},
    ]{p2ErrorSmooth.csv}
    \caption{$L^2$-error, residual and experimental order of convergence (eoc). Example \ref{subsec:detSmooth}.}
    \label{table:detSmooth} 
\end{table}
\subsection{Deterministic Adaptivity:  Sod Shock Tube Problem} \label{subsec:SodShockdet}
In this numerical experiment we apply the adaptive spatial mesh refinement from \algoref{algo:hAdaptive} to the Sod shock tube problem. 
The Riemann data for this problem is given by
\begin{equation*} \label{eq:initialEulerSod}
\begin{alignedat}{1}
 \rho(t=0,x,y) &= \begin{cases}
1, \qquad &x < 0.5 \hspace*{2cm}\\
0.125, \qquad &x \geq 0.5,
\end{cases}\\
m(t=0,x,y) &= 0,\\
E(t=0,x,y) &= \begin{cases}
2.5, \qquad &x < 0.5,\\
0.25, \qquad &x\geq 0.5.
\end{cases}
\end{alignedat}
\end{equation*}

\begin{figure}[htb!]
  \centering
  \settikzlabel{fig:sod1}
   \settikzlabel{fig:sod2} 
  \includegraphics[width=2.5\figurewidth, height=1.5\figureheight]{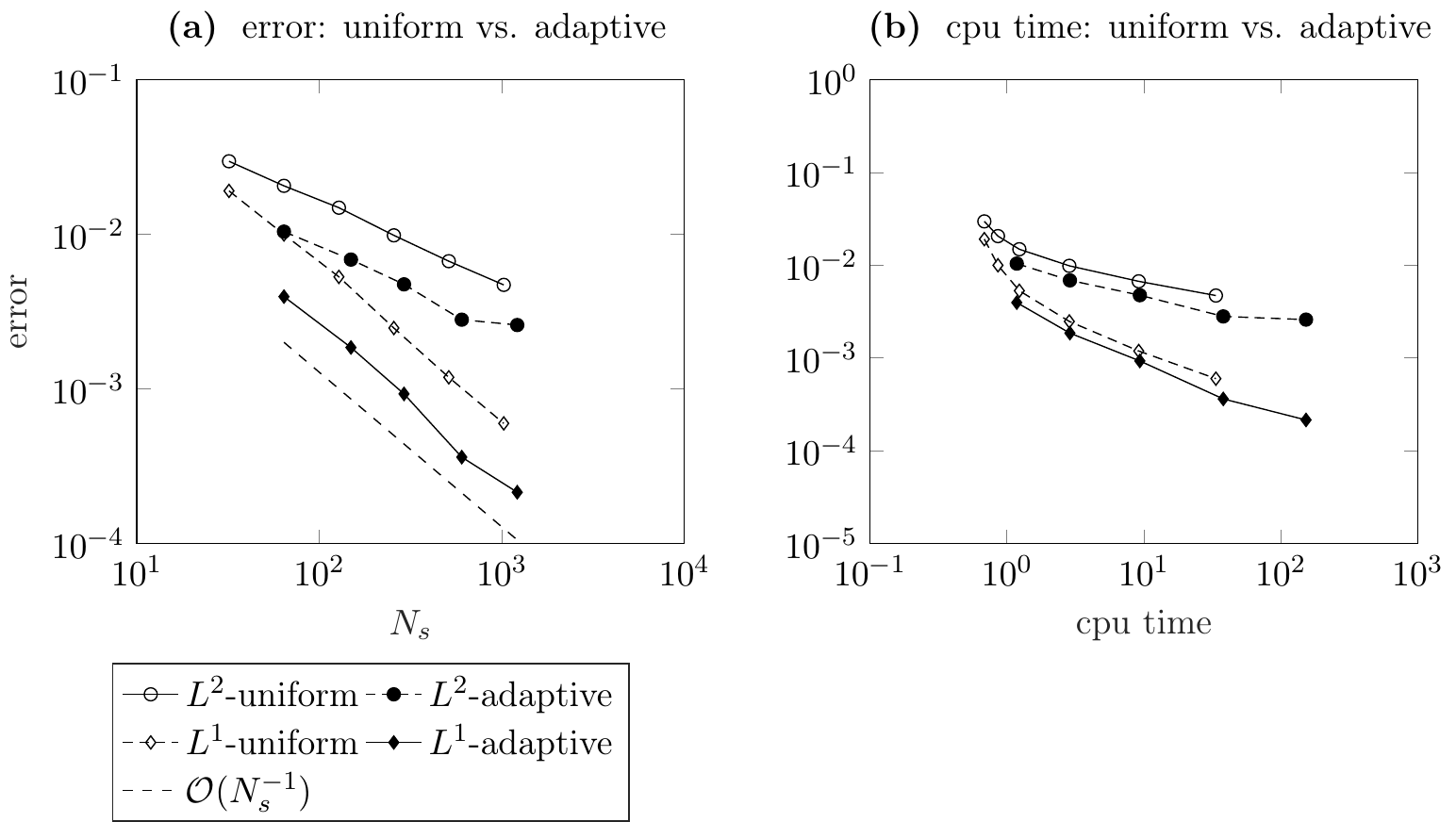}
   \caption{Error plot for the deterministic Sod shock tube problem. Example \ref{subsec:SodShockdet}}
  \label{fig:sod}
\end{figure}

The numerical solution is computed on the domain $\Lambda=[0,1]$ up to $T=0.2$ using the Lax-Friedrichs flux \eqref{eq:LF} and a DG polynomial degree 
of two. In this example we use exact boundary conditions. 
In \figref{fig:sod1} we compare the $L^1(\Lambda)$- and $L^2(\Lambda)$-error at time $T$ between the numerical solution and the exact solution obtained with 
an exact Riemann solver \cite{Backus2017}. 
We can see that for the same number of spatial cells $\nCells$, the numerical error obtained with the adaptive numerical algorithm 
is smaller than for the uniform mesh refinement. The adaptive algorithm is also computationally more efficient than
the uniform algorithm, which can be seen in the error vs. cpu time plot in \figref{fig:sod2}. 
\begin{remark}
As discussed in \rmref{rm:residuals} (iv), $\residual{det}$ scales with $\frac{1}{h}$ in the vicinity of shocks
and contact discontinuities, i.e., it blows up under spatial mesh refinement in these areas.
Thus, if we view the residual as an error indicator, it severely over-estimates the error so that it is to be called 
``inefficient'' in these areas, 
according to the nomenclature of e.g. \cite{Verfuerth2013}.
From the point of view of mesh adaptation however, refinement based on $\residual{det}$ 
leads to a reasonable refinement strategy that yields a considerable improvement
in error decay compared to uniform mesh refinement (cf. \figref{fig:sod}). In particular,
$\residual{det}$  precisely captures the positions of rarefaction waves, contact discontinuities 
and shocks.

Over-estimating the error at discontinuities 
leads to maximal refinement at discontinuities and some refinement strategies for
hyperbolic conservation laws suggest
a maximal refinement close to shocks \cite{PuppoSemplice2011}. 
\end{remark}
\subsection{A Stochastic Problem with Smooth Solution} \label{subsec:SmoothStoch}
In this section we focus on the scaling properties of the stochastic residual for a
one- and two-dimensional random space $\Xi$ and a random flux function.

\subsubsection{A One-Dimensional Random Space, $q$-Refinement} \label{subsubsec:smooth1D}
We modify the exact solution from 
\secref{subsec:detSmooth} in the following way, 
\begin{align}\label{eq:stochref}
	\begin{pmatrix}
	\rho(t,x,y) \\
	 m(t,x,y)  \\
	E(t,x,y)  
	\end{pmatrix}=
	\begin{pmatrix}
	2+ 0.1\cos(4\pi (x-yt)) \\[0.1cm]
	\Big(2+ 0.1\cos(4\pi (x-yt)) \Big)\Big(1+0.1\sin(4\pi (x-yt))\Big) \\
	\Big(2+ 0.1\cos(4\pi (x-yt)) \Big)^2
	\end{pmatrix}.
\end{align}
The numerical solution is computed on $\Lambda=[0,1]_{\mathrm{per}}$ up to $T=0.2$, the uncertainty $y$ stems from an uniform distribution, 
i.e. $\xi\sim \mathcal{U}(0,8)$. We consider two different spatial meshes consisting
of $\nCells = 32$ and 512 elements respectively, 
a DG polynomial degree of $p=2$ and we use the Lax-Wendroff numerical flux \eqref{eq:LW}.
In this  numerical example we globally approximate the function \eqref{eq:stochref} in $\Xi$,
i.e. we increase the polynomial degree $q$ and consider one ME.
\begin{figure}[htb!]
  \centering
  \settikzlabel{fig:stochSmooth_global}
  \includegraphics[width=2.5\figurewidth, height=1.5\figureheight]{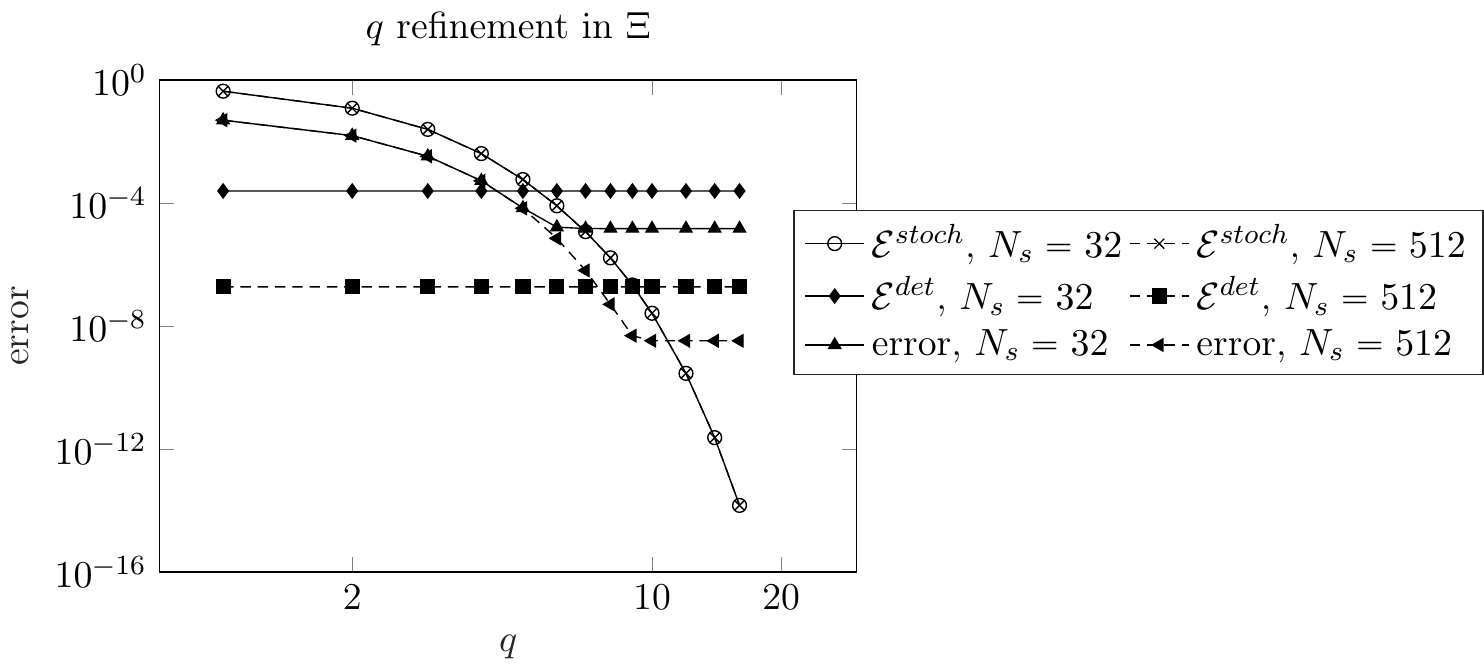}
  \caption{Error plot for stochastic smooth problem. Example \ref{subsubsec:smooth1D}}
  \label{fig:stochSmooth}
\end{figure}

\figref{fig:stochSmooth} shows the behavior of the error and the spatial, resp. stochastic residual, when we globally interpolate the smooth function
\eqref{eq:stochref}.
We see that the stochastic residual $\stochResidual$ exhibits spectral convergence. Also the numerical error exhibits 
spectral convergence until it starts to stagnate because of the spatial resolution error.
This is the correct behavior of the stochastic residual as we are globally increasing the polynomial degree in the random space and, therefore, expect 
spectral convergence with increasing polynomial degree.
We also observe that the exponential convergence of $\stochResidual$ is not altered by a finer or coarser space discretization.
Moreover, the deterministic residual $\detResidual$ is unaffected by the increasing resolution in the random space, 
which we expect from the residual's splitting into a space-time and a stochastic part.

\subsubsection{Mesh Refinement in $\Xi$ and Random Flux Function}  \label{subsubsec:smooth2D}
In this example we examine the scaling properties of $\stochResidual$ under
mesh refinements for a two-dimensional random space $\Xi\subset \R^2$.
 We consider the same smooth function as in \secref{subsubsec:smooth1D},
\begin{align}\label{eq:stochref2D}
	\begin{pmatrix}
	\rho(t,x,y_1) \\
	 m(t,x,y_1)  \\
	E(t,x,y_1)  
	\end{pmatrix}=
	\begin{pmatrix}
	2+ 0.1\cos(4\pi (x-y_1t)) \\[0.1cm]
	\Big(2+0.1\cos(4\pi (x-y_1t)) \Big)\Big(1+0.1\sin(4\pi (x-y_1t))\Big) \\
	\Big(2+ 0.1\cos(4\pi (x-y_1t)) \Big)^2
	\end{pmatrix}.
\end{align}
with $\xi_1\sim \mathcal{U}(0,8)$. Moreover, we consider a random adiabatic constant. We assume  that $\gamma=\xi_2 \sim \mathcal{U}(1.4,1.6)$ and
thus the flux function is also random. The randomness of the adiabatic-constant corresponds to considering a gas mixture of uncertain composition.
The numerical solution is computed on $\Lambda=[0,1]_{\mathrm{per}}$ up to $T=0.2$. We consider a fixed spatial mesh consisting of $\nCells = 32$ elements.
For the ME-SC method we perform a linear and a quadratic interpolation, i.e. $q\in \{1,2\}$.

\begin{figure}[htb!]
  \centering
  \settikzlabel{fig:stochSmooth_ME}
  \includegraphics[width=2.5\figurewidth, height=1.5\figureheight]{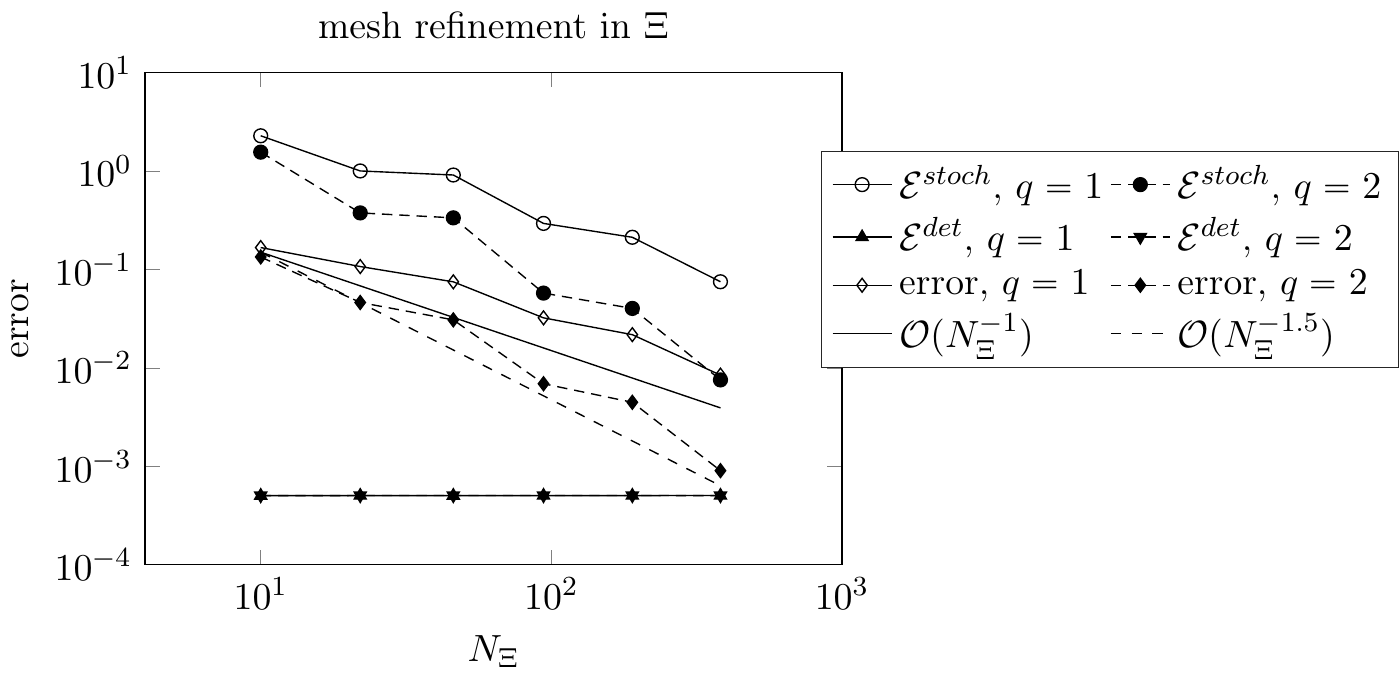}
  \caption{Error plot for stochastic smooth problem. Example \ref{subsubsec:smooth2D}}
  \label{fig:stochSmooth2D}
\end{figure}

\figref{fig:stochSmooth2D} illustrates the behavior of the stochastic residual $\stochResidual$, when we consider a local interpolation, i.e., when we
consider the ME method from \secref{subsec:ME}. We observe that for a local linear and quadratic interpolation, i.e. $q\in \{1,2\}$, the stochastic residual converges approximately with the 
expected rate of convergence, which is $(q+1)/2$, cf. \cite{WanKarniadakis2006a}. 
Like for the $q$-refinement in \secref{subsubsec:smooth1D}, the deterministic residual $\detResidual$ stays constant, when we increase the number of MEs.
 
\subsection{Stochastic Adaptivity: Stochastic Problem with Discontinuous Solution} \label{subsec:DiscStoch}
We apply the stochastic adaptive \algoref{algo:stochAdaptive} without spatial adaptivity
to a solution which has a discontinuity in the random variable 
and compare the results with uniform space-stochastic mesh refinements.
We therefore consider the following discontinuous function,
\begin{align*}
	\begin{pmatrix}
	\rho(t,x,y_1,y_2) \\
	 m(t,x,y_1,y_2)  \\
	E(t,x,y_1,y_2)  
	\end{pmatrix}=
	\begin{pmatrix}
	1+ A(y_1,y_2)\cos(4\pi (x-y_1t)) \\[0.1cm]
	\Big(1+ A(y_1,y_2)\cos(4\pi (x-y_1t)) \Big)\Big(1+0.1\sin(4\pi (x-y_1t))\Big) \\
	\Big(1+ A(y_1,y_2)\cos(4\pi (x-y_1t)) \Big)^2
	\end{pmatrix},
\end{align*}
where $$A(y_1,y_2)
=\begin{cases} 0.1, \text{ if } y_1^2+y^2\leq 0.5^2 \\
0.2, \text{ else } .
\end{cases}$$ 
is a discontinuous amplitude. For the spatial domain $\Lambda=[0,1]_{\mathrm{per}}$ we use $\nCells = 32$  elements and a DG polynomial degree of two.
The solution is computed up to $T=0.2$ using the Lax-Wendroff numerical flux \eqref{eq:LW}
and for the uncertainty we assume that $\xi_1, \xi_2\sim \U(0,1)$. For the ME-SC method we consider a linear interpolant, i.e. $q=1$. 

\begin{figure}[htb!]
  \centering
  \settikzlabel{fig:stochAdaptive_Error}
  \settikzlabel{fig:stochAdaptive_wall time}
  \includegraphics[width=2.5\figurewidth, height=1.5\figureheight]{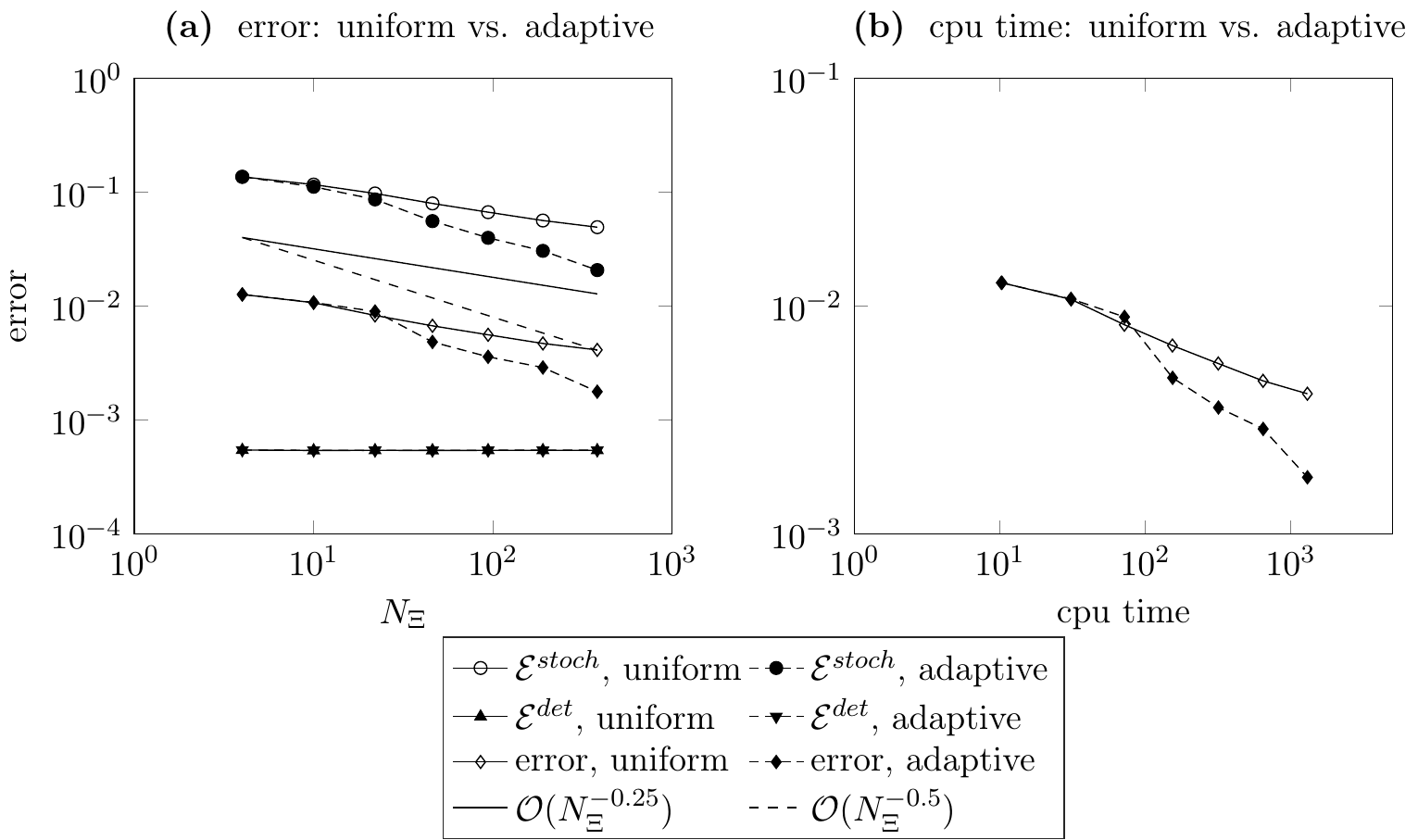}
  \caption{Error plot for discontinuous stochastic problem. Example \ref{subsec:DiscStoch}}
  \label{fig:stochAdaptive}
\end{figure}

In \figref{fig:stochAdaptive_Error} we plot the error and the spatial resp. stochastic residual versus the number of MEs 
and in \figref{fig:stochAdaptive_wall time} we show the error of the uniform and adaptive method versus cpu time. 
In \figref{fig:stochAdaptive_Error}  we can observe that for the uniform stochastic refinement, both the error and the
stochastic residual $\stochResidual$ converge with a rate of approximately $1/4$. This is in accordance with what we expect
when interpolating a two-dimensional discontinuous function.
For the adaptive refinement the error and the residual exhibit a rate of convergence
of approximately $1/2$. The advantage of the stochastic adaptive algorithm is also 
reflected in \figref{fig:stochAdaptive_wall time}, where we reach an error reduction in significantly less time compared to uniform refinement.

\subsection{Space-Stochastic Adaptivity: An Uncertain Riemann Problem} \label{subsec:spaceStochSod}
Finally, we assess the efficiency of the space-stochastic adaptive algorithm by considering a random Riemann Problem. The initial data for this
problem reads as follows
\begin{align*}
\begin{alignedat}{1}
 \rho(t=0,x,y) &=1 \\
m(t=0,x,y) &= 
\begin{cases}
y_1, & x \leq 0.5 \\
y_2, &x > 0.5 \\
\end{cases} \\
p(t=0,x,y) &= 1,
\end{alignedat}
\end{align*}
where $\xi_1,\xi_2 \sim \mathcal{U}(-1,1)$ and $\Lambda=[0,1]$.
We compare the space-stochastic adaptive \algoref{algo:stochAdaptive} with uniform refinement,
both in 
physical and random space. For this problem we use the Lax-Friedrichs numerical flux \eqref{eq:LF} and for the uniform spatial mesh we consider $\nCells=512$ spatial elements. As for the Sod Shock Tube problem in 
\secref{subsec:SodShockdet} we prescribe exact boundary conditions.
For the adaptive algorithm we always start on a spatial mesh consisting of $\nCells=256$ elements.
The DG polynomial degree is two and we consider a linear interpolation in the random space, i.e. $q=1$. The solution is computed up to $T=0.2$.
The error is measured in the expected value rather than the ${\leb{2}{\Xi}{\Leb^2(\Lambda)}}$-norm. Note that we do not have an exact solution at hand for this problem, but due  to Jensen's inequality,
\begin{align} \label{eq:numericalError}
\|\E(u(T,\cdot,\cdot))-\E(\numSol^{\ntCells}(\cdot,\cdot))\|_{L^2(\Lambda)}^2 & \leq \E\|u(T,\cdot,\cdot)- \numSol^{\ntCells}(\cdot,\cdot)\|_{L^2(\Lambda)}^2  \\
& = \|u(T,\cdot,\cdot)-\numSol^{\ntCells}(\cdot,\cdot)\|_{\leb{2}{\Xi}{\Leb^2(\Lambda)}}^2. \nonumber
\end{align}
The reference expectation $\E(u(T,\cdot,\cdot))$ is computed using a Monte-Carlo method with an exact Riemann solver
with 500000 samples.
\begin{figure}[htb!]
  \centering
  \settikzlabel{fig:sodStochdaptive_Error}
  \settikzlabel{fig:sodStochdaptive_wall time}
  \includegraphics[width=2.5\figurewidth, height=1.5\figureheight]{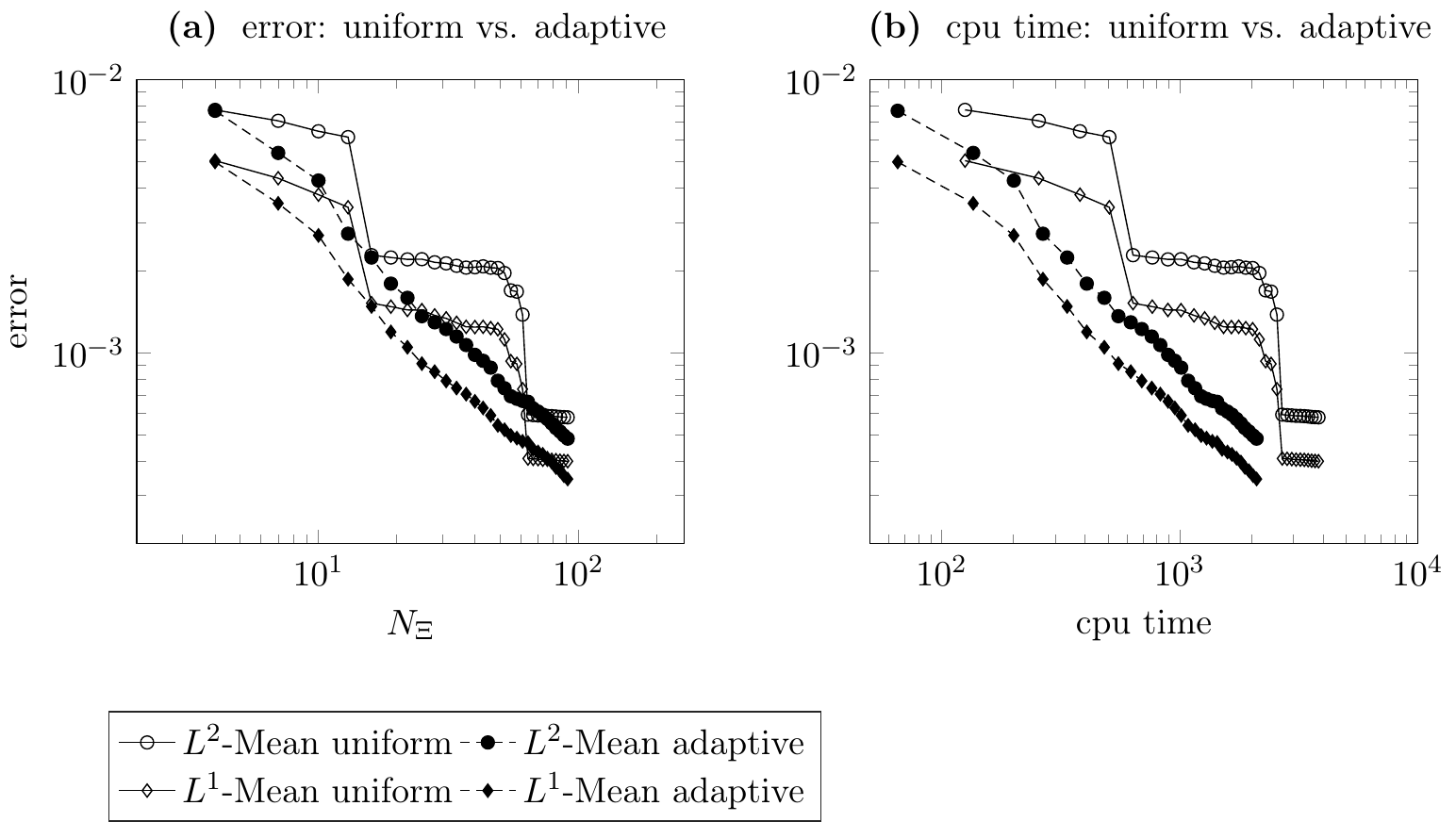}
  \caption{Error plot uncertain Riemann problem. Example \ref{subsec:spaceStochSod}}
  \label{fig:RiemannAdaptive}
\end{figure}
In \figref{fig:sodStochdaptive_Error} we show the numerical error as in \eqref{eq:numericalError} and we also consider
the error $\|\E(u(T,\cdot,\cdot))-\E(\numSol^{\ntCells}(\cdot,\cdot))\|_{L^1(\Lambda)}$ for an increasing number of MEs,
i.e. for increasing $\MEElements$. We can see that the adaptive algorithm decreases the error considerably faster than
the uniform refinement. 
This is also depicted in the cpu time vs. error plot (\figref{fig:sodStochdaptive_wall time}), 
where we can see that the adaptive algorithm
reaches an absolute error in significantly less computational time than the uniform algorithm. 
This demonstrates, in particular,
the efficiency of our proposed method.

\bibliographystyle{siam}
\bibliography{mybib.bib}
\appendix

\section{RKDG Method and Space-Time Reconstructions} \label{appendix:A}
In the following we describe the space-time discretization of \eqref{eq:DIVP} that we use
and the space-time reconstruction of the numerical solution. 
\subsection{The RKDG Method}
For the space-time discretization of \eqref{eq:DIVP} we use a Runge--Kutta DG method (RKDG). 
We recall the DG spatial discretization as for example in \cite{cockburn1998}. 
For the ease of presentation we neglect the dependence of the flux $\Flux$, the spatial mesh and the DG spaces on the collocation points $\{y_\globalIdx\}_{\globalIdx \in \cK}$.

Let $\T:= \{ I_k\}_{k=0}^{{\nCells-1}}$, $I_k:=(x_k,x_{k+1})$ be a quasi-uniform triangulation of $\Lambda=[0,1]_{\mathrm{per}}$. We set $h_k= (x_{k+1}-x_k)$, $h_{\max}= \max \limits_k h_k$, $h_{\min}= \min \limits_k h_k$ for the spatial mesh and identify $x_0=x_{\nCells}$ to account
for periodic boundary conditions.
Further let $0=t_0<t_1<\ldots< t_{\ntCells}=T$ be a  temporal decomposition of $[0,T]$ and define $\Delta t_n :=( t_{n+1}-t_n)$, $\Delta t =\max \limits_n ~\Delta t_n$. With each time-interval $(t_n,t_{n+1}]$ we associate a (possibly different) partition $\T_n$ and associated  DG space 
\begin{align*}
\DGn{p}:= \{v:\Lambda \to \R^m~|~ v\mid_{I} \in \P_p(I,\R^m),~ \text{ for all } I\in \Tn \}.
\end{align*} 
With  $\mathcal{L}_{\DGn{p}}$ we denote the $L^2$-projection mapping into the DG space $\DGn{p}$.

Following \cite{DednerOhlberger2008} we call the function $\numSol$  a  generalized  semi-discrete DG
approximation of \eqref{eq:DIVP} if it satisfies for $\numSol^{-1}:=\mathcal{L}_{V_{p,0}^s} \stochSol^0$
the following equations.
For every $n=0,\ldots,\ntCells$, $\numSol^n \mid_{[t_n,t_{n+1}]} \in C^1((t_n,t_{n+1}); \DGn{p})\cap C^0([t_n,t_{n+1}]; \DGn{p})$,
\begin{equation}\tag{DG} 
\begin{aligned} \label{eq:semidiscrete}
\numSol^n(t^n)=&~\mathcal{L}_{\DGn{p}} \numSol^{n-1}(t^n), \\
\sum \limits_{i=0}^{\nCells-1} \int \limits_{x_i}^{x_{i+1}}  \partial_t  \numSol^n \cdot \psi_h~\mathrm{d}x  
=& \sum \limits_{i=0}^{\nCells-1} \int \limits_{x_i}^{x_{i+1}} L_h^n(\numSol^n) \cdot   \psi_h~\mathrm{d}x   \quad \forall \psi_h \in \DGn{p},
\end{aligned}
\end{equation}
where
$ L_h^n:\DGn{p} \to \DGn{p}$
is defined by
\begin{multline}
 \sum \limits_{i=0}^{\nCells-1} \int \limits_{x_i}^{x_{i+1}}  L_h^n(v) \cdot \psi_h~\mathrm{d}x
 =
 \sum \limits_{i=0}^{\nCells-1} \int \limits_{x_i}^{x_{i+1}} \Flux(v) \cdot \partial_x  \psi_h~\mathrm{d}x     \vspace*{8pt}
\\ -\sum \limits_{i=0}^{\nCells-1} \numFlux(v(x_i^-),v(x_i^+) )\cdot  \dbracei \psi_h \dbraceo_i, \qquad \forall v, \psi_h \in \DGn{p}.
\end{multline}
The numerical solution $\numSol$ is defined through $\numSol(0):= \numSol^{-1}$ and $\numSol \mid_{(t_n,t_{n+1}]}:=\numSol^n\mid_{(t_n,t_{n+1}]}$.

Here,  $\numFlux: \R^m\times \R^m \to \R^m$ denotes a numerical flux, 
the spatial traces are defined as $\psi (x^{\pm}):= \lim \limits_{h\searrow 0} \psi(x \pm h)$ and $\dbracei \psi_h \dbraceo_i:=(\psi_h(x_i^-)-\psi_h(x_i^+))$ are jumps.

The initial-value problem \eqref{eq:semidiscrete} can now be solved numerically by any single- or multi-step method. 
We focus on $K$-th order Runge-Kutta time-step methods as in \cite{CockburnShu2001,Ketcheson2008}.
 Furthermore, $\Lambda \Pi_h : \R^{m} \to \R^{m}$ is the TVBM minmod slope limiter from \cite{CockburnShu2001}. 
 Then, the complete $S$-stage time-marching algorithm for given $n$-th time-iterate $\numSol^n(t_n)\in \DGn{p} $ reads as follows:
\begin{algorithm}[H] 
\caption{TVBM Runge--Kutta Time-Step}
\begin{algorithmic}[1]
\State Set $\numSol^{(0)}$ = $\numSol^n(t_n)$.
\For{$j=1,\ldots, S$}
\State Compute: $ \numSol^{(j)} =  \Lambda \Pi_h\Big( \sum \limits_{l=0}^{j-1} \alpha_{jl} w_h^{jl}\Big), \quad w_h^{jl}=\numSol^{(l)} + \frac{\beta_{jl}}{\alpha_{jl}} \Delta t_n L_h^n(\numSol^{(l)}).$
\EndFor
\State Set $\numSol^{n}(t_{n+1}) = \numSol^{(S)}.$
\end{algorithmic}
\label{algo:TVBMRungeKutta}
\end{algorithm}
The parameters $\alpha_{jl}$ satisfy the conditions $\alpha_{jl}\geq 0$, $\sum \limits_{l=0}^{j-1} \alpha_{jl}=1$ , and if $\beta_{jl} \neq 0$, then $\alpha_{jl} \neq 0$ for all $j=1,\ldots, S$, $l=0,\ldots,j$.
\subsection{Space-Time Reconstruction}
Our analysis relies on reconstructing the numerical solution $\{\numSol^n\}_{n=0}^{\ntCells}$ to a 
Lipschitz continuous function in space and time. We structure the reconstruction process as follows:
\begin{enumerate}
  \item \textbf{Computation of a temporal reconstruction $\reconst{t}$:}
  
  We first compute the temporal reconstruction as proposed in \cite{GiesselmannDedner16}.

  Let $\{\numSol^0, \ldots, \numSol^{\ntCells} \}$ 
be a sequence of approximate solutions of \eqref{eq:semidiscrete} at points $\{t_n\}_{n=0}^{\ntCells}$ in time,
where we assume that all approximate solutions are interpolated onto a reference mesh $\T$, which is a common refinement of all meshes. With $V_p^s$ we denote the DG space associated with $\T$, 
hence $\numSol(t_{n}) \in V_p^s$ for all
$n=0,\ldots, \ntCells$.

For the reconstruction in time we define the spaces of piecewise polynomials in time of degree $r$ by
\begin{align*}
V_r^t((0,T);V_p^s):=\{w:[0,T]\to V_p^s~|~w\mid_{(t_n,t_{n+1})} \in \P_r((t_n,t_{n+1}),V_p^s) \}.
\end{align*}
Using Hermite interpolation on each time interval $[t_n,t_{n+1}]$, 
we construct the temporal reconstruction $\hat{\stochSol}^t \in V_r^t((0,T);V_p^s)$.
  \item \textbf{Computation of a space-time reconstruction $\reconst{st}$ using the time reconstruction $\reconst{t}$:}
  
  With the temporal reconstruction  $\reconst{t}$ at hand, we define the space-time reconstruction  $\reconst{st}$ of the DG-solutions of \eqref{eq:semidiscrete}. 
The analysis in \cite{GiesselmannDedner16} requires numerical fluxes $\numFlux$ which admit a special representation.
In particular, there needs to exist a  locally Lipschitz function $w: \U \times \U \to \U$, with the additional property $w(\stochSol,\stochSol)=\stochSol$,
such that  $\numFlux$ can either be expressed as
\begin{align}\label{assum:flux}
\numFlux(\stochSol,v)=\Flux(w(\stochSol,v)), \quad \forall \stochSol,v \in  \U.
\end{align}
or as 
\begin{align}\label{assum:flux2}
\numFlux(\stochSol,v)=\Flux(w(\stochSol,v))- \mu(\stochSol,v;h)h^{\nu}(v-\stochSol), \quad \forall \stochSol,v \in  \U,
\end{align}
where $\nu \in \N$ and for some matrix-valued function $\mu$, which has the property that for any compact $K\subset \U$
there exists a $\mu_K>0$, such that $|\mu(\stochSol,v;h)|\leq \mu_K (1+ \frac{|v-\stochSol|}{h})$, for $h$ small enough. 
\begin{remark}
For our numerical computations we consider the following numerical fluxes.
\begin{itemize}
\item  The
Lax-Wendroff flux: $\numFlux(\stochSol,v)=\Flux(w(\stochSol,v))$ with $w(\stochSol,v)= \frac{\stochSol +v}{2} - \frac{\Delta t}{2 h}(\Flux(\stochSol)-\Flux(v))$,  satisfies \eqref{assum:flux}.
\item The Lax-Friedrichs flux :  $G(u,v) = \frac{1}{2}\Big(F(u)+F(w) \Big) + \lambda(w-u)$ satisfies \eqref{assum:flux2}, with
$\nu = 0$, $w(\stochSol,v):=\frac{1}{2}(\stochSol+v)$ and $\mu(\stochSol,v;h):= \lambda I - \frac{\Flux(\stochSol)-2 \Flux(w(\stochSol,v)) + \Flux(v)}{2 |v- u|^2}\otimes(\stochSol-v)$.
\end{itemize}
\end{remark}

We define the spatial reconstruction which is applied to the temporal reconstruction $\hat{\stochSol}^t(t,\cdot)$ for each $t\in (0,T)$ using the function $w$ (cf. \cite{GiesselmannDedner16, GiesselmannMakridakisPryer15}). 
\begin{definition}[Space-time reconstruction]\label{def:stReconstruct}
Let $\hat{\stochSol}^t$ be the temporal reconstruction of a sequence $\{\stochSol_h^n\}_{n=0}^{\ntCells}$ of solutions of the fully discrete scheme of \emph{\eqref{eq:semidiscrete}} 
using a numerical flux satisfying \eqref{assum:flux} or \eqref{assum:flux2}. 
The space-time reconstruction $\reconst{st}(t,\cdot) \in V_{p+1}^s$ is defined as the solution of
\begin{align*}
\sum \limits_{i=0}^{\nCells-1} \int \limits_{x_i}^{x_{i+1}}(\reconst{st}(t,\cdot)- \reconst{t}(t,\cdot)) \cdot \psi ~\mathrm{d}x &=0 \quad \forall \psi \in V_{p-1}^s,
\\ 
\reconst{st}(t,x_k^{\pm})&= w(\reconst{t}(t,x_k^-),\reconst{t}(t,x_k^+)) \quad \forall~ k=0,\ldots,\nCells.
\end{align*}
\end{definition}
We have the following property of the space-time reconstruction.
\begin{lemma}[\hspace{-1pt}\cite{GiesselmannDedner16}, Lemma 24]
Let $\reconst{st}$ be the space-time reconstruction from Definition \ref{def:stReconstruct}. 
For each $t\in (0,T)$, the function $\reconst{st}(t,\cdot)$ is well defined. Moreover,
\begin{align*}
\reconst{st}\in W_\infty^1((0,T);V_{p+1}^s\cap C^0(\Lambda)).
\end{align*}
\end{lemma}
\end{enumerate}
.

\end{document}